\documentclass[a4paper,12pt]{amsart}
\usepackage[utf8]{inputenc}
\usepackage{amsthm}
\usepackage{amsmath}
\usepackage{bm}
\usepackage{enumitem}
\usepackage{amsfonts}
\usepackage{amssymb}
\usepackage{mathtools}
\usepackage{appendix}
\usepackage{mathrsfs}
\usepackage{setspace}
\usepackage{comment}
\usepackage{xcolor}
\usepackage{todonotes}
 \usepackage{ulem}
\usepackage{thmtools} % solves problem of shared counters in statements and autoref
\usepackage[foot]{amsaddr} %%
\usepackage{appendix}
\RequirePackage[a4paper,top=2.54cm,bottom=2.54cm,left=1.90cm,right=1.90cm,%
                headsep=1em,includehead,includefoot]{geometry}
\usepackage[pdfdisplaydoctitle,colorlinks,breaklinks,urlcolor=blue,linkcolor=blue,citecolor=blue]{hyperref} 
\usepackage[nameinlink,capitalise]{cleveref}%must always be the last
\providecommand{\U}[1]{\protect\rule{.1in}{.1in}}
\newtheorem{theorem}{Theorem}

\newtheorem{lemma}[theorem]{Lemma}

\newtheorem{remark}[theorem]{Remark}

\title[Hasegawa-Mima equation in bounded domain with singular density]
{Hasegawa-Mima equation in bounded domain with singular density}

\author[Franco Flandoli]{Franco Flandoli}
\address{Scuola Normale Superiore, Piazza dei Cavalieri, 7, 56126 Pisa, Italy}
\email{\href{mailto:franco.flandoli at sns.it}{franco.flandoli at sns.it}}
\author[Yassine Tahraoui]{Yassine Tahraoui}
\address{Scuola Normale Superiore, Piazza dei Cavalieri, 7, 56126 Pisa, Italy}
\email{\href{mailto:yassine.tahraoui at sns.it}{yassine.tahraoui at sns.it}}
\date\today

\begin{document}
	\begin{abstract} We investigate  the well-posedness of  Hasegawa-Mima equation  (HME) in bounded domain with Dirichlet boundary condition and  singular density under  different regularity assumptions on the data.  Our approach relies on the coupling of a fourth-order regularization of the HME, the spectral Galerkin method and an appropriate regularization of the singular density. Uniqueness holds in the class of solutions with very mild singularities à la Yudovich or with bounded  densities.\end{abstract}

    %%%done in \cite{FlaTah24}
	\maketitle
\tableofcontents
\section{Introduction}
The Hasegawa-Mima equation (HME) has been introduced by Akira Hasegawa and
Kunioki Mima in their 1977 paper \cite{HM} and popularized by several works,
see for instance Horton and Hasegawa \cite{HH}. The equation describes a
confined plasma subject to a very strong constant magnetic field, under
certain approximations. The equation is sometimes called the
Charney-Hagesawa-Mima equation because Jule G. Charney developed a similar
model for climate studies in 1948 \cite{Charney}. But the Coriolis term in
Charney model is different from the term $\log n_{0}$ of Hasegawa-Mima
equation. Coriolis varies rather gradually up to the poles; on the contrary,
$-\log n_{0}$ is quite flat and moderate in the interior of the domain
occupied by the plasma but assumes very high values near the boundary. The
shear $\nabla^{\perp}\log n_{0}$ is extremely strong near the boundary and may
contribute with other causes to the existence of very important structures
(the pedestal, the boundary zonal flows \cite{Diamond}). Even more, if we
assume that $n_{0}$ vanishes at the boundary, the term $\log n_{0}$ diverges
(and $\nabla^{\perp}\log n_{0}$ even more); we shall call this case
\textit{singular}, below. In this respect, HME is quite different from more
classical fluid mechanic equations, like Euler equations, including the
Charney model, although many analogies exist \cite{HH}.\\

We consider the 2D HME in a bounded open  domain $D$ 
  with smooth boundary $\partial D$\footnote{For example,  $C^4$-boundary is sufficient but one can relax that  by  using some approximation. However, we  will comment where the regularity of boundary $\partial D$ plays a role. }, for instance the
disk $B\left(  0,R_{0}\right)  $ of center zero and radius $R_{0}>0$. The
equation, in the unknown $\varphi:\left[  0,T\right] \times D  \rightarrow
\mathbb{R}$, reads 
\begin{align}\label{HM-eqn}
	\begin{cases}
	&\partial_t(\varphi-\Delta \varphi)+\nabla^\perp \varphi\cdot\nabla(\log (n_0)-\Delta \varphi)=0 \quad  \text{ in } [0,T]\times D, \\
	&\varphi_{\vert t=0}=\varphi_0, \qquad  \varphi_{\vert \partial D}=0,
	\end{cases}
	\end{align}
    where $\nabla^\perp \varphi=(-\partial_2\varphi,\partial_1\varphi)$  and $T>0$. Here $n_{0}:D\rightarrow\mathbb{R}$ is a positive function,
independent of time; for instance, when $D=B\left(  0,R_{0}\right)  $ we may
think of a radial density $n_{0}\left(  x\right)  =\widetilde{n}_{0}\left(
\left\vert x\right\vert \right)  $ where $\widetilde{n}_{0}$ is positive and
decreasing in $[0,R_{0})$. We pay particular attention to the assumptions on
$\log n_{0}$ below. 
Making more precise what we said above, we call singular the case when $\log
n_{0}$ is not bounded (in the example $D=B\left(  0,R_{0}\right)  $,
$n_{0}\left(  x\right)  =\widetilde{n}_{0}\left(  \left\vert x\right\vert
\right)  $ with decreasing $\widetilde{n}_{0}$, the singular case is when
$\lim_{r\rightarrow R_{0}}\widetilde{n}_{0}\left(  r\right)  =0$). Our main
purpose is proving a general existence result in the singular case; and two
uniqueness results, one in the singular case - but applicable only to very
mild singularities - and the other one in the non-singular case.\\

Having in mind  the 2D Euler equation  for incompressible fluids in the stream function form, it is worth recalling that  this later does not have a characteristic spatial scale, unlike  \eqref{HM-eqn}, which have a consequence on the turbulence spectra. Another important fact is that \eqref{HM-eqn}  is not divergence-free.  Moreover,  the inhomogeneity of \eqref{HM-eqn} due to the presence of $\log(n_0)$ implies that \eqref{HM-eqn} admits linear waves, which provide modifications to the cascade process. For more details, see \cite{HH}.

\subsection{Physical reasons for the singular case}
The motivation for paying attention to the possible singularity of $\log
n_{0}$ comes from the specific example of confined fusion plasma models of
tokamaks. The plasma occupies a three dimensional domain, a torus in real
tokamaks, a cylinder with 2D basis $D$, called the poloidal section, in the
simplified model of HME. The magnetic field is assumed to be strong and
constant, oriented perpendicularly to the plane of $D$. The plasma is composed
of two species, ions and electrons, almost neutral, namely with almost the
same density of ions and electrons;\ the function $n_{0}$ is precisely an
average smooth value of such densities, free of small local fluctuations. Both
ions and electrons densities $n_{i}\left(  x,t\right)  $ and $n_{e}\left(
x,t\right)  $ have, however, small local variations with respect to
$n_{0}\left(  x\right)  $, producing a nonzero electrostatic potential
$\varphi\left(  x,t\right)  $, the main variable of the HME. All these fields
are assumed to depend only on the poloidal coordinates and the electric field
$E=-\nabla\varphi$ is therefore poloidal only. The main reason to recall these
few elements (the reader may see \cite{HH} for more details) is to insist on
the meaning of $n_{0}$, an average background ion (or electron) density. Does
it take the value zero (hence $\log n_{0}$ is singular) or not, at the
boundary of the domain occupied by the plasma? Recall that the plasma is
suspended in an empty space, not attached to solid boundaries, hence it is
natural to expect that density decreases near the boundary for reasons of
pressure. Good confinement however may keep it strictly positive up to the
boundary. Very recent experimental data on advanced fusion devices maybe show
a very small positive value (see for instance figure 3d of \cite{Ding}), but
first it is difficult to decide whether it is really positive, second it may
be convenient for the mathematical analysis of properties to assume it zero,
to detect a sharper behavior.\\

Moreover, to make the problem even more complicated, one should decide where
the boundary of the region occupied by the plasma is. The particle density
figures of all experimental and numerical papers on confined fusion plasma use
the convention that $r=1$ is the normalized plasma minor radius; in figure 3d
of \cite{Ding}, the value of the density at $r=1$ is certainly very small, but
maybe strictly positive. However, there is a region for $r$ between $1$ and
roughly 1.1-1.2 still occupied by particles and there the density goes to
zero, see figures 3, 7, 10 and others of \cite{Liu}. If we include that region
in the domain, the density at the boundary is certainly zero. The positive
value at $r=1$ may have to do with the pedestal and with the separating
surface, depending on experiments and conventions. 
Therefore it is important to have rigorous information on the well posedness
in both the singular and non singular case.
%%%% on related models
\subsection{Literature, strategy of proof and  content of the paper}
Before presenting the content of this paper and without seeking to be exhaustive,  the equation \eqref{HM-eqn} in 2D has been already studied in other works
but in simpler geometries (like the torus) and without the potential
difficulties of a singular $\log n_{0}$, see for instance \cite{Guo,Grauer}, and also   \cite{Cao} for the  3D case. The emphasis on the regularity of $\log n_{0}$ seems to be new. On the other hand, V.I. Yudovich, in \cite{Yudovich1963,Yudovich1995uniqueness}, used the stram function form of the 2D Euler equation to prove existence and uniqueness with very weak singularities, which is the best uniqueness result known so far for the 2D Euler equation. There are other proofs based on flow approaches \cite{Crippa,Kato} of the 2D Euler equation and other recent work on non-uniqueness with a wider class of data \cite{Brue2023nonuniqueness}.\\

For the convenience of the reader, let us explain our strategy. We are interested in the  hyperbolic-elliptic  equation \eqref{HM-eqn} with different regularity assumptions on the data. We need first to introduce an appropriate regularization, which serves as a basis for rigorously justifying the derivation of certain estimates related to \eqref{HM-eqn}. This regularization needs to respect the boundary conditions and ensure the desired estimates, which motivates the introduction of the fourth-order PDE \eqref{HM-eqn-modified}. Due to the boundary condition and the form of \eqref{HM-eqn-modified},  in first step we prove a well-posedness result of \eqref{HM-eqn-modified} by using Galerkin method via the construction of an appropriate spectral basis \eqref{basis-eigen}. Once the well-posedness of \eqref{HM-eqn-modified} is established, we have access to a large class of test functions depending on the solution of \eqref{HM-eqn-modified}. However, due to the weak regularity of $\log n_{0}$, we need to regularize it appropriately, as we are in a bounded domain with a Dirichlet boundary condition. This is the aim of \autoref{lemma-approximation} if $2\leq p \leq +\infty$ and of \autoref{lemma-approximation-p-less-2} if $\frac{4}{3}<p<2.$ Therefore, we can use test functions that depend (nonlinearly) on the solution of \eqref{HM-eqn-modified} and the regularization of $\log n_{0}$. Note that we need to use different test functions in the cases $2\leq p < +\infty$ , $p=+\infty$ and $\frac{4}{3}<p<2$. Thus, a uniform estimate with respect to the regularization parameters is obtained by using  \eqref{HM-eqn-modified}, and we  pass to the limit via a compactness argument to construct a solution to HME \eqref{HM-eqn}. The above procedure also allows us to obtain a uniform bounds that depends only on the norms of $\log n_{0}$ and the initial data, and under  appropriate assumptions we can prove a uniqueness result as long as the $p$-norm of the solution grows appropriately (Yudovich class-like).

\subsubsection*{Structure of the paper}
The manuscript is organized as follows: In \autoref{SectionI}, we present some preliminaries and  collect the main results.
 Then, we prove the existence and uniqueness to  a fourth-order  regularization of Hasegawa-Mima equation \eqref{HM-eqn-modified}  in \autoref{section-regularized}, which is used in our analysis of \eqref{HM-eqn}. \autoref{section-exis-uniq-Yod-1} is devoted the proof of existence and uniqueness of solution to \eqref{HM-eqn} with  $\log (n_0) \in L^p(D)$ and  initial data $\varphi_0\in  W^{2,p}(D)$ when $2\leq p< +\infty$.
 \autoref{Section-bounded-density-result} is devoted the proof of existence and uniqueness of solution to \eqref{HM-eqn} with  $\log (n_0) \in L^\infty(D)$ and  initial data $\varphi_0-\Delta \varphi_0\in L^\infty(D)$.
In \autoref{Section-p-less-2}, we  present  the proof of existence  of solution to \eqref{HM-eqn} with  $\log (n_0) \in L^p(D)$ and  initial data $\varphi_0\in  W^{2,p}(D)$ when 
 $\frac{4}{3}< p<2$.

\section{Preliminaries and main results}\label{SectionI}

\subsection{Preliminaries}
\subsubsection{Notation and function spaces}
Let $D\subset \mathbb{R}^{2}$,   $(L^{p}(D))_{1\leq p\leq +\infty}$ denotes the Lebesgue space, $(W^{k,p}(D))_{1\leq p\leq +\infty}^{k\in \mathbb{N}^*}$ stands for the Sobolev space endowed with the norm $\|\cdot\|_{W^{k,p}}$. We set $H^{k}(D):=W^{k,2}(D),$ $k\in \mathbb{N}^*$ and $H^1_0(D)$ is the closure of the space of smooth functions with compact support with respect to the $H^{k}(D)$-norm.\\

    Denote by $V=H^1_0(D)$ and 
     $W=H^1_0(D)\cap H^2(D)$. The spaces $V$ and  $W$ are endowed with the following inner products.
     \begin{align*}
         (u,v)_V&=(u,v)+(\nabla u,\nabla v), \qquad \forall u,v \in V\\
         (u,v)_W&=(u,v)_V+((I-\Delta)u,(I-\Delta)v), \qquad\forall u,v \in W, \end{align*}
     where $(\cdot,\cdot)$ denotes the $L^2(D)$-inner product.  It is clear that $(u,v)_V=((I-\Delta)u,v)$ for any $u,v \in W.$ We will use also $\Vert \cdot\Vert_V$ and  $\Vert \cdot\Vert_W$  to denote the corresponding norms, which are equivalent to the usual norms of $H^1(D)$ and $  H^2(D)$ spaces, respectively.   We use the following notation: $X^\prime$ denotes the dual space of a given Banach space $X$ and the duality pairing is denoted by $\langle \cdot,\cdot\rangle$, $\Vert \cdot \Vert_p$ denotes the norm in the usual Lebesgue space $L^p(D)$ and   $\eta=(\eta_1,\eta_2)$  denotes the external normal on $\partial D.$ Recall that   $\partial_t$  denote the partial derivative with respect to $t$, namely $\partial_t:=\dfrac{\partial}{\partial t}.$\\

     In order to construct the solution to \eqref{HM-eqn} and derive rigorously some \textit{a priori} estimates, we need to construct a particular Galerkin basis compatible with the boundary condition and also ensures the derivation of appropriate estimate at the level of finite dimensional approximation. This is the  aim of the next subsection.
     \subsubsection{Construction of  Galerkin basis}
Since $W\hookrightarrow V$ is compact, there exists a set of eigenfunctions $\{e_i\}_{i\in \mathbb{N}}$ (see \textit{e.g.} \cite[Thm. 4.1]{Spectral}) such that 
    \begin{align}\label{basis-eigen}
     e_i\in W:   (u,e_i)_W=\lambda_i(u,e_i)_V, \quad \forall u\in W, \quad  i\in \mathbb{N}, \quad \lambda_i>0  \text{ and } \lambda_i \nearrow +\infty.
    \end{align}
  Moreover,  $\{e_i\}_{i\in \mathbb{N}}$ forms an orthonormal basis of $V$ and  $\{\tilde{e}_i=\dfrac{e_i}{\sqrt{\lambda_i}}\}_{i\in \mathbb{N}}$ is  an orthonormal basis of $W.$ On the other hand,  $e_i$ satisfies $\Delta e_{i_{\vert \partial D}}=0$ for any $i\in \mathbb{N}.$  Indeed, let $u\in C^\infty_c(D)$  and expand  \eqref{basis-eigen}, we derive  (by applying Green's formula)
  \begin{align}\label{EDP-eigen}
      (I-\Delta)e_i+ (I-\Delta)(e_i-\Delta e_i)=  \lambda_i(I-\Delta) e_i \text{ in } L^2(D) , \quad \forall i\in \mathbb{N} .
  \end{align}
  Next, let $u\in W$. Then, we  expand again  \eqref{basis-eigen} and use \eqref{EDP-eigen}  to obtain 
  \begin{align*}
   \langle \Delta e_i, \nabla u\cdot \eta\rangle_{(H^{1/2})^{\prime},H^{1/2}}=0 , \qquad   \forall u \in W.
  \end{align*}
  
  Since $u$ is arbitrary  element of $W$ and by continuity of the trace operator from $H^1(D)$ to $H^{1/2}(\partial D)$, we deduce that  $\Delta e_i=0$ on $\partial D$ in $(H^{1/2})^{\prime}$-sense. Thus, $e_i$ solves
 \begin{align}\label{BVP-1}
	\begin{cases}
	&(I-\Delta)e_i+ (I-\Delta)(e_i-\Delta e_i)=  \lambda_i(I-\Delta) e_i \text{ in } D, \\
	&  e_{i_{\vert \partial D}}=0,\quad \Delta e_{i_{\vert \partial D}}=0 \text{ on } \partial D.
	\end{cases}
	\end{align}
 Now, due the fact the  domain $D$ has a smooth boundary then  $\{e_i\}_{i\in \mathbb{N}}$, which satisfies \eqref{BVP-1}, are smooth functions until the boundary thanks  to  the classical elliptic regularity theory of PDEs, see \textit{e.g.} \cite[Thm. 9.25]{BrezisBook}. In particular, if $D$ has  $C^4$-boundary then  $\{e_i\}_{i\in \mathbb{N}} \subseteq H^4(D)$, which sufficient to justify rigorously all the computations in this paper.\\
 
    Now,  consider $W_n=\text{span}\{e_1, \cdots,e_n\}$ and let us introduce the following projection operator $P_n:W^\prime \to W_n$ defined as follows
\begin{align*}
    P_n:W^\prime &\to W_n, \quad 
    u\mapsto P_nu=\sum_{i=1}^{i=n}\langle u,e_i\rangle e_i.
\end{align*}
Note that the restriction of $P_n$ on  $W, V$ defined as $$P_nu=\sum_{i=1}^{i=n}(u,\tilde{e}_i)_W\tilde{e}_i, \quad P_nu=\sum_{i=1}^{i=n}(u,e_i)_Ve_i,$$ respectively, is an orthogonal projection and  thus $\Vert P_n\Vert_{L(W,W_n)}\leq 1$ and $\Vert P_n\Vert_{L(V,W_n)}\leq 1.$
\subsection{Main results}
In this part, we collect the main results.
   \begin{theorem}\label{THM1-0}
  Let $\varepsilon>0$. Under the assumption $(H)$,   there exists  a unique   $ \varphi \in C([0,T];W) \cap L^2(0,T;Y)  $
solution to \eqref{HM-eqn-modified} in the following sense
     \begin{align}\label{DEF-form}
 &       \langle \partial_t(\varphi-\Delta\varphi),\psi\rangle_{V^\prime,V}+\int_D[\varepsilon\nabla(\varphi-\Delta\varphi+g)\cdot \nabla \psi -(g-\Delta \varphi)\nabla^\perp\varphi\cdot \nabla \psi]dx=0, \quad \forall \psi \in V.
    \end{align}  
   \end{theorem}
\begin{proof}
    See \autoref{THM1} and \autoref{section-regularized}.
\end{proof}
Next, we use \autoref{THM1-0} to study the existence and uniqueness results of \eqref{HM-eqn}. We distinguish three cases: the case $2\leq p< +\infty$, $p=+\infty$ and the case $\frac{4}{3}< p<2.$
  \begin{theorem}\label{TH-existence-HM-0}
  Let $2\leq p< +\infty$.  Assume that $\log (n_0) \in L^p(D)$
and      $\varphi_0-\Delta \varphi_0\in L^p(D)$ with $\varphi_0=0$ on $\partial D$.  Then there exist at least $\varphi\in L^\infty(0,T;V\cap W^{2,p}(D))$, solution  to \eqref{HM-eqn}, such that
    \begin{itemize}
        \item $\varphi \in  C([0,T];W^{1,q}(D)) \text{ for any } q \text{ finite}$  and 
        $\varphi \in C([0,T];V\cap W^{2,p}(D)-w)$\footnote{$X-w$ refers to the space $X$ equipped with the corresponding weak topology.}.
   \item $\varphi$ satisfies $\varphi(0)=\varphi_0$ and  solves \eqref{HM-eqn} in the following sense   
      \begin{align}\label{Def-sol_HM}
 &       \langle \partial_t(\varphi-\Delta\varphi),\psi\rangle_{W^\prime,W}-\int_D(\log (n_0)-\Delta \varphi)\nabla^\perp\varphi\cdot \nabla \psi dx=0, \quad \forall \psi \in W.
    \end{align}  
    \item There exists $K_*>0$ independent of $p$ and $\varphi$ such that
   \begin{align}\label{inequ-p-uniq-*}
   \sup_{t\in[0,T]}   \Vert \varphi(t)-\Delta\varphi(t)\Vert_p \leq  K_*(\Vert \varphi_0-\Delta\varphi_0\Vert_p+\Vert \log (n_0)\Vert_p).
 \end{align}
     \end{itemize}
 
     \end{theorem}
     \begin{remark}
         Recall that  $\varphi_0-\Delta \varphi_0\in L^p(D)$ such that $\varphi_0=0$ on $\partial D$ is equivalent to  $\varphi_0\in H^{1}_0(D)\cap W^{2,p}(D)$ thanks to Agmon-Douglis-Nirenberg theorem, see \textit{e.g.}  \cite[Thm. 9.32]{BrezisBook}.
     \end{remark}
     Now, we introduce \textit{Yudovich space} associated with growth function $\theta$, denoted by $Y^\theta$. Namely
\begin{align}\label{Yudovich-space}
    Y^\theta(D)=\big\{ f\in \bigcap_{1\leq p<+\infty}\hspace{-0.4cm}\big(L^{p}(D)\big): \quad \Vert f\Vert_{Y^\theta}=\sup_{1\leq p<+\infty} \dfrac{\Vert f \Vert_{L^p}}{\theta(p)} <\mathbf{A} \big\}, \quad \mathbf{A}>0
\end{align}
where $\theta$ is positive and  non-decreasing function defined on $[1,+\infty[$. We associate with the function $\theta$ the following
\begin{align}\label{function-Osgood}
	\Phi_\theta(r)=	\begin{cases}
 \inf \{\frac{2}{\varepsilon}\theta(\frac{2}{\varepsilon} ): \quad \varepsilon \in ]0,1/2[ \} &\text{ for } r\in [0,1[, \\
\inf \{\frac{2}{\varepsilon}\theta(\frac{2}{\varepsilon} )r^{\varepsilon/2}: \quad \varepsilon \in ]0,1/2[ \} &\text{ for } r\in [1,+\infty[.
	\end{cases}
	\end{align}
    
     \begin{theorem}(Uniqueness à la Yudovich)\label{THm-uniqu-YuD}
    Assume that $\varphi_0,\log(n_0) $ and  $\Phi_\theta$ given by \eqref{function-Osgood}  satisfy \begin{align}
 \label{Osgood-condition}    &\int_0^1\dfrac{dr}{r\Phi_\theta(1/r)}=+\infty,
 \\
 \label{assumption-uniquness}   &\varphi_0-\Delta\varphi_0 \in Y^\theta(D) \text{ and } \log(n_0)\in Y^\theta(D).
\end{align}
 Then, there exists a unique solution of \eqref{HM-eqn} in the sense of  \autoref{TH-existence-HM-0}.
\end{theorem}
\begin{proof}
   For the proof of \autoref{TH-existence-HM-0} and \autoref{THm-uniqu-YuD}, see \autoref{section-exis-uniq-Yod-1}.
\end{proof}
  \begin{theorem}\label{TH-existence-HM-bounded-section0}
   Assume that $\log (n_0) \in L^\infty(D)$
and       $\varphi_0\in Y$ with  $\varphi_0-\Delta \varphi_0\in L^\infty(D)$.  Then there exits a unique $\varphi\in L^\infty(0,T;V\cap W^{2,p}(D))$ for any $1\leq p<+\infty$, solution  to \eqref{HM-eqn}, such that
    \begin{itemize}
        \item $\varphi \in  C([0,T];W^{1,q}(D)) \text{ for any } q \text{ finite}$  and 
        $\varphi \in C([0,T];V\cap W^{2,p}(D)-w).$   \item $\varphi$ satisfies $\varphi(0)=\varphi_0$ and  solves \eqref{HM-eqn} in the following sense    \begin{align}\label{Def-sol_HM-2}
 &       \langle \partial_t(\varphi-\Delta\varphi),\psi\rangle_{W^\prime,W}-\int_D(\log (n_0)-\Delta \varphi)\nabla^\perp\varphi\cdot \nabla \psi dx=0, \quad \forall \psi \in W.
    \end{align}  
    \item  $\varphi-\Delta\varphi \in L^\infty(D\times[0,T]) $ and the following uniform bound holds
   \begin{align*}  \Vert \varphi-\Delta\varphi \Vert_\infty \leq 2(\Vert \varphi_0-\Delta\varphi_0\Vert_\infty+\Vert \log (n_0)\Vert_\infty).
 \end{align*}
     \end{itemize}
     \end{theorem}
     \begin{proof}
         See \autoref{Section-bounded-density-result}.
     \end{proof}
   \begin{remark}
     By interpolation we deduce that the solution given in \autoref{TH-existence-HM-0} and \autoref{TH-existence-HM-bounded-section0} $\varphi \in C([0,T];H^s(D)),$ for any $0<s<2$.
     \end{remark}
     \begin{remark} It is worth mentioning that the  assumption $\varphi_0\in Y$ in \autoref{TH-existence-HM-bounded-section0} can   be relaxed to $\varphi_0 \in W$ by considering the following approximation in the  proof of \autoref{TH-existence-HM-bounded-section0}. Namely, for any $\varepsilon>0$ consider $\phi_\epsilon \in Y\cap H^4(D)$, the unique solution to the following problem \eqref{regu-initiial}:
        \begin{align}\label{regu-initiial}
	\begin{cases}
	&(I-\Delta)\phi_\varepsilon -\varepsilon\Delta(\phi_\varepsilon-\Delta \phi_\varepsilon)=  (I-\Delta) \varphi_0 \text{ in } D, \\
	&  \phi_{\varepsilon_{\vert \partial D}}=0,\quad \Delta \phi_{\varepsilon_{\vert \partial D}}=0 \text{ on } \partial D.
	\end{cases}
	\end{align}  
    Similarly to \autoref{lemma-approximation}, one shows that $(I-\Delta)\phi_\varepsilon \to (I-\Delta)\varphi_0$ in $L^2(D)$ and \begin{align*}
        \Vert  (I-\Delta)\phi_\varepsilon\Vert_\infty \leq \Vert (I-\Delta)\varphi_0\Vert_\infty.
    \end{align*}
     \end{remark}
        \begin{theorem}\label{thm-p-less2}
       Let $\frac{4}{3}< p<2$.  Assume that $\log (n_0) \in L^p(D)$
and      $\varphi_0-\Delta \varphi_0\in L^p(D)$ with $\varphi_0=0$ on $\partial D$.  Then there exist at least $\varphi\in L^\infty(0,T;V\cap W^{2,p}(D))$, solution  to \eqref{HM-eqn}, such that
    \begin{itemize}
        \item $\varphi \in  C([0,T];W^{1,r}(D)) \text{ for any } r <\frac{2p}{2-p} $  and 
        $\varphi \in C([0,T];V\cap W^{2,p}(D)-w)$.
   \item $\varphi$ satisfies $\varphi(0)=\varphi_0$ and  solves \eqref{HM-eqn} in the following sense   
      \begin{align}\label{Def-sol_HM-p-less-2}
 &       \langle \partial_t(\varphi-\Delta\varphi),\psi\rangle_{W^\prime,W}-\int_D(\log (n_0)-\Delta \varphi)\nabla^\perp\varphi\cdot \nabla \psi dx=0, \quad \forall \psi \in W^{2,\frac{p}{p-1}}(D)\cap V .
    \end{align}  
    \item There exists  $\mathcal{K}_*>0$ independent of $p$ and $\varphi$ such that
   \begin{align*}
   \sup_{t\in[0,T]}   \Vert \varphi(t)-\Delta\varphi(t)\Vert_p \leq  \mathcal{K}_*(\Vert \varphi_0-\Delta\varphi_0\Vert_p+\Vert g\Vert_p).
 \end{align*}
     \end{itemize}
 \end{theorem}
 \begin{proof} See \autoref{Section-p-less-2}.
 \end{proof} 
 \begin{remark}
     If $1\leq p<  \frac{4}{3}$,   \textit{a priori} $\Delta \varphi\nabla^\perp\varphi \notin L^1_{loc}(D)$ and one needs to use the notion of  renormalized solution. This is outside the scope of this work and we refer  \textit{ e.g.} to  \cite{Diperna-Lions1989}.
 \end{remark}
\section{Regularized Hasegawa-Mima equation}\label{section-regularized}
In this section, let us introduce the following Banach space
 $$Y=\{ v \in W: \quad  \nabla \Delta v \in (L^2(D))^2,\quad \Delta v_{\vert \partial D}=0\}$$ equipped with the norm
  $\Vert u\Vert_Y^2=\Vert u\Vert_W^2+\Vert \nabla(I-\Delta)u\Vert_2^2$. Note that an element of $Y$ implies that its Laplacian belongs to $V$ and therefore the trace $\Delta v_{\vert \partial D}=0$ is understood in $H^{1/2}$-sense.
\begin{remark}
 It is not difficult to see that the standard $H^3$-norm  is  equivalent to the $Y$-norm.
\end{remark}
 Let $\varepsilon >0$ and  consider the following equation modified version of $\eqref{HM-eqn}.$
\begin{align}\label{HM-eqn-modified}
	\begin{cases}
	&\partial_t(\varphi-\Delta \varphi)+\nabla^\perp \varphi\cdot\nabla(g-\Delta \varphi)=\varepsilon \Delta(\varphi-\Delta \varphi+g) \quad   \text{ in } ]0,T]\times D, \\
	&\varphi_{\vert t=0}(x)=\varphi_0(x), \qquad  \varphi_{\vert \partial D}=0,\quad \Delta \varphi_{\vert \partial D}=0.
	\end{cases}
	\end{align}
  First, it is clear that the solution $\varphi$ depends on  $\varepsilon$. In order to not overload the notation in this section, we omit  dependencies on $\varepsilon$ here.
    In this section, we assume the following
    \begin{itemize}
        \item[(H)]   Assume that  $\varphi_0 \in Y$ and  $g\in W$.
    \end{itemize}
     
   % \begin{remark}
         %\end{remark}
   \begin{theorem}\label{THM1}
  Let $\varepsilon>0$. Under the assumption $(H)$, it holds
  \begin{itemize}
      \item    There exists  a unique   $ \varphi \in C([0,T];W) \cap L^2(0,T;Y)  $
solution to \eqref{HM-eqn-modified} in the following sense
     \begin{align}\label{DEF-form}
 &       \langle \partial_t(\varphi-\Delta\varphi),\psi\rangle_{V^\prime,V}+\int_D[\varepsilon\nabla(\varphi-\Delta\varphi+g)\cdot \nabla \psi -(g-\Delta \varphi)\nabla^\perp\varphi\cdot \nabla \psi]dx=0, \quad \forall \psi \in V.
    \end{align}  
\item Moreover,  let $\varphi_i$ be the solution to \eqref{DEF-form} associated with initial datum $\varphi_0^i, i=1,2$, it holds
  \begin{align*}
     &\sup_{t\in[0,T]}\Vert \varphi(t)\Vert_W^2+\varepsilon\int_0^T\Vert \nabla (\varphi-\Delta\varphi)\Vert_2^2 ds\leq \widetilde{\mathbf{C}}[\Vert \varphi_0^1-\varphi_0^2\Vert_W^2] \exp{\big(\dfrac{\tilde{\mathbf{C}}}{\varepsilon}\int_0^T(\Vert g\Vert_4^2+\Vert \varphi_2\Vert_Y^2+\Vert  \varphi_1\Vert_W)ds\big)}.
\end{align*}
where $\varphi=\varphi_1-\varphi_2$,  and   $\widetilde{\mathbf{C}}>0$ independent of $ \varepsilon, \varphi_i; i=1,2.$
\end{itemize}
   \end{theorem}
  
     \begin{remark}
     One can see that the proof of \autoref{THM1} works  in the same way for the 3D case but we focus only on  2D, since our interest is \eqref{HM-eqn} in 2D.
   \end{remark}
   %%%\begin{proof}
   \subsection{Proof of \autoref{THM1}} We divide the proof into five steps.
   \subsubsection*{I. Galerkin approximation}
Let $n\in \mathbb{N}$,  denote by $\varphi_0^n=P_n \varphi_0$ and recall that: 
   \begin{align*}
   \Vert P_n \varphi_0\Vert_W\leq \Vert  \varphi_0\Vert_W, \quad \Vert P_n \varphi_0\Vert_V\leq \Vert  \varphi_0\Vert_V.    
   \end{align*}
We  introduce the Galerkin approximation for \eqref{HM-eqn}. For that,  define $$\varphi_n(t,x):= \sum_{i=1}^{n}c_i(t)e_i(x).$$
By construction, we have
\begin{align}\label{BC-approxi}
    \varphi_n= \Delta \varphi_n=0 \text{ on } \partial D.
\end{align}
Let $i\in\{1,\cdots, n\},$  we consider  the following ODE system
\begin{align*}
    \partial_t(\varphi_n-\Delta\varphi_n,e_i)+(\nabla^\perp\varphi_n\cdot\nabla(g-\Delta \varphi_n),e_i)=\varepsilon (\Delta(\varphi_n-\Delta\varphi_n+g),e_i);\quad \varphi_n(0)=\varphi_0^n,
\end{align*}
which is equivalent to
\begin{align}\label{Galerkin}
\partial_t(\varphi_n,e_i)_V+(\nabla^\perp\varphi_n\cdot\nabla(g-\Delta \varphi_n),e_i)=\varepsilon (\Delta(\varphi_n-\Delta\varphi_n+g),e_i),\quad \varphi_n(0)=\varphi_0^n.
\end{align}
Note that \eqref{Galerkin} is a system of nonlinear differential equation with the unknown $c\in \mathbb{R}^n$ satisfying 
\begin{align}\label{systemODE}
  \dfrac{dc}{dt}+Bc+ c\otimes c: D_i=F, \quad c_i(0)=(\varphi_0,e_i)_V , \quad  i=1,\cdots, n,
\end{align}
where  
\begin{align*}
 B_{i,j}= (\nabla^\perp e_j\cdot\nabla g-\varepsilon \Delta (e_j-\Delta e_j),e_i),    D_{i,j,l}=-(\nabla^\perp e_j\cdot\nabla(\Delta e_l),e_i)  \text{ and }F_i=\varepsilon(\Delta g,e_i).   
\end{align*}
 By classical theory of ODEs, there exists $t_n>0$  and $c\in C([0,t_n],\mathbb{R}^n)$ solution of \eqref{systemODE}. Consequently, there exists $\varphi_n \in C([0,t_n],W_n)$ solution to \eqref{Galerkin}.
   \subsubsection*{II. Uniform estimates of $(\varphi_n)_n$ with respect to $n$} Let $t\in [0,t_n]$, after multiplying \eqref{Galerkin} by $c_i$ and summing from $i=1$ to $n$, we get
\begin{align}
    \partial_t\Vert \varphi_n\Vert_V^2+(\nabla^\perp\varphi_n\cdot\nabla(g-\Delta \varphi_n),\varphi_n)=\varepsilon (\Delta (\varphi_n-\Delta \varphi_n+g),\varphi_n), \quad \varphi_n(0)=\varphi_0^n.
\end{align}
Since $\text{div}\nabla^\perp\varphi_n=0$, we have 
\begin{align}\label{est-1-step-1} 
    (\nabla^\perp\varphi_n\cdot\nabla \Delta \varphi_n,\varphi_n)
   &=-\int_D \nabla^\perp\varphi_n\cdot\nabla\varphi_n\Delta \varphi_n dx+\int_{\partial D}  \nabla^\perp\varphi_n\cdot \eta\varphi_n\Delta \varphi_n  d\sigma=0,
\end{align}
due to the fact that  $\nabla^\perp\varphi_n\cdot\nabla\varphi_n=0$ and the boundary conditions \eqref{BC-approxi}. Similarly, we get
\begin{align}\label{est-1-step-2}
      (\nabla^\perp\varphi_n\cdot\nabla g,\varphi_n)&=-\int_D g\nabla^\perp\varphi_n  \cdot\nabla\varphi_n dx+\int_{\partial D}  \nabla^\perp\varphi_n\cdot \eta\varphi_ng  d\sigma=0.
\end{align}
Next, by using Green's formula and \eqref{BC-approxi}, we get \begin{align}\label{est-1-step-3}
    &(\Delta (\varphi_n-\Delta \varphi_n),\varphi_n)=-\Vert  \Delta \varphi_n\Vert_2^2-\Vert \nabla \varphi_n\Vert_2^2 \notag \\ &\text{ and }  (\Delta g,\varphi_n)=( g,\Delta\varphi_n)\leq \dfrac{1}{2}\Vert g\Vert_2^2+\dfrac{1}{2}\Vert \Delta\varphi_n\Vert_2^2.
\end{align}
Hence \begin{align}\label{estimate-V}
\forall t\in [0,t_n]:\Vert \varphi_n(t)\Vert_V^2+&\varepsilon\int_0^t[\dfrac{1}{2}\Vert  \Delta \varphi_n\Vert_2^2+\Vert \nabla \varphi_n\Vert_2^2]ds\notag \\&\leq \Vert \varphi_0^n\Vert_V^2+\dfrac{\varepsilon t_n}{2}\Vert g\Vert_2^2\leq \Vert\varphi_0\Vert_V^2+\dfrac{t_n}{2}\Vert g\Vert_2^2,\quad \text{ for any } \varepsilon \in [0,1].
\end{align}
$\bullet$ In order to derive $W$-estimate,  we use the relation between the inner products in $V$ and $W$, namely \eqref{basis-eigen}. First,  we need to regularize some terms in \eqref{Galerkin}. Namely,  denote by
\begin{align}\label{regularized-term}
  G_n=\nabla^\perp\varphi_n\cdot\nabla (g-\Delta \varphi_n)-\varepsilon \Delta(\varphi_n-\Delta\varphi_n+g) \in L^2(D)
\end{align} and  let $g_n \in W$ be the unique solution (see \textit{e.g.} \cite[Thm. 9.25]{BrezisBook}) of 
\begin{align*}
\begin{cases}
	& g^n-\Delta g^n=G_n  \text{ in }  D, \\
    &g^n_{\vert \partial D}=0,
	\end{cases}    
\end{align*}
which satisfies the following equality (the weak formulation)
\begin{align}\label{auxiliary}
    (g^n,v)_V=(G_n,v), \quad \forall v \in W.
\end{align}
By using \eqref{Galerkin} and since $e_i\in W$, we have
\begin{align}\label{Galerkin-2}
    \partial_t(\varphi_n,e_i)_V+(g^n,e_i)_V=0, \quad i=1,\cdots n;\quad \varphi_n(0)=\varphi_0^n.
\end{align}
Now, we multiply \eqref{Galerkin-2} by $\lambda_i$ and we use \eqref{basis-eigen} to get
\begin{align}\label{Galerkin-W}
    \partial_t(\varphi_n,e_i)_W+(g^n,e_i)_W=0, \quad i=1,\cdots n;\quad \varphi_n(0)=\varphi_0^n.
\end{align}
We recall that $$\varphi_n:= \sum_{i=1}^{i=n}c_ie_i=\sum_{i=1}^{i=n}(\varphi_n,e_i)_Ve_i= \sum_{i=1}^{i=n}(\varphi_n,\tilde{e}_i)_W\tilde{e}_i=\sum_{i=1}^{i=n}\tilde{c}_i\tilde{e}_i.$$
By multiplying \eqref{Galerkin-W} by $\dfrac{1}{\sqrt{\lambda_i}}$ first, then multiplying by $\tilde{c}_i$ and sum for $i=1$ to $n$, we obtain
\begin{align}
    \partial_t\Vert\varphi_n\Vert_W^2+(g^n,\varphi_n)_W=0, \quad \varphi_n(0)=\varphi_0^n.
\end{align}
By using the definition of the inner product in $W$, we get
\begin{align}
    \partial_t\Vert\varphi_n\Vert_W^2+(G_n,\varphi_n)+(G_n,\varphi_n-\Delta \varphi_n)=0, \quad \varphi_n(0)=\varphi_0^n.
\end{align}
By using $\nabla^\perp\varphi_n\cdot\nabla\varphi_n=0$, the boundary condition \eqref{BC-approxi} and Green's formula, we get
\begin{align*}
    (G_n,\varphi_n-\Delta \varphi_n)&=\int_D  \nabla^\perp\varphi_n\cdot\nabla (g-\Delta \varphi_n) (\varphi_n-\Delta \varphi_n) dx-\varepsilon\int_D \Delta(\varphi_n-\Delta\varphi_n+g)(\varphi_n-\Delta \varphi_n) dx\\
    &=\int_D  \nabla^\perp\varphi_n\cdot\nabla g (\varphi_n-\Delta \varphi_n) dx+\varepsilon\int_D \vert\nabla(\varphi_n-\Delta\varphi_n)\vert^2 dx+\varepsilon\int_D \nabla g\cdot \nabla(\varphi_n-\Delta \varphi_n) dx.
\end{align*}
On the other hand, by using Hölder's inequality we get
\begin{align*}
    &\vert \int_D  \nabla^\perp\varphi_n\cdot\nabla g (\varphi_n-\Delta \varphi_n) dx+\varepsilon\int_D \nabla g\cdot \nabla (\varphi_n-\Delta \varphi_n) dx\vert \\& \leq
    \Vert \nabla^\perp\varphi_n\Vert_4\Vert \nabla g\Vert_4\Vert \varphi_n-\Delta \varphi_n\Vert_2+\dfrac{\varepsilon}{2}\Vert \nabla g\Vert_2^2+\dfrac{\varepsilon}{2}\Vert \nabla(\varphi_n-\Delta\varphi_n)\Vert_2^2\\
    & \leq C\Vert g\Vert_W\Vert \varphi_n\Vert_W^2+\dfrac{\varepsilon}{2}\Vert \nabla g\Vert_2^2+\dfrac{\varepsilon}{2}\Vert \nabla(\varphi_n-\Delta\varphi_n)\Vert_2^2
\end{align*}
where $C>0$ depends only on the embedding $W\hookrightarrow W^{1,4}(D).$ 
Summarizing, we get
\begin{align*}
    \partial_t\Vert\varphi_n\Vert_W^2+(G_n,\varphi_n)+\dfrac{\varepsilon}{2}\Vert \nabla(\varphi_n-\Delta\varphi_n)\Vert_2^2\leq C\Vert g\Vert_W\Vert \varphi_n\Vert_W^2+\dfrac{\varepsilon}{2}\Vert \nabla g\Vert_2^2.
\end{align*}
By using \eqref{est-1-step-1}-\eqref{est-1-step-3}, we get
\begin{align}
\notag  &\sup_{t\in[0,t_n]}\Vert\varphi_n(t)\Vert_W^2+\int_0^{t_n}[\dfrac{\varepsilon}{2}\Vert  \Delta \varphi_n\Vert_2^2+\varepsilon\Vert \nabla \varphi_n\Vert_2^2+\dfrac{\varepsilon}{2}\Vert \nabla(\varphi_n-\Delta\varphi_n)\Vert_2^2]ds\\&\leq \notag\Vert\varphi_0^n\Vert_W^2+ C\Vert g\Vert_W\int_0^{t_n}\Vert \varphi_n\Vert_W^2 ds+\dfrac{\varepsilon t_n}{2}[\Vert \nabla g\Vert_2^2  +\Vert g\Vert_2^2] \\ \notag
&\leq \Vert\varphi_0\Vert_W^2+ \dfrac{\varepsilon t_n}{2}\Vert  g\Vert_V^2 +C\Vert g\Vert_W\int_0^{t_n}\Vert \varphi_n\Vert_W^2 ds.
\end{align}
 By using Gronwall's lemma, one gets
\begin{align}\label{estimate-W_T}
  \sup_{t\in[0,t_n]}  \Vert \varphi_n(t)\Vert_W^2+\int_0^{t_n}[\dfrac{\varepsilon}{2}\Vert  \Delta \varphi_n\Vert_2^2+\varepsilon\Vert \nabla \varphi_n\Vert_2^2+\dfrac{\varepsilon}{2}\Vert \nabla(\varphi_n-\Delta\varphi_n)\Vert_2^2]ds\notag\\ \leq  [\Vert \varphi_0\Vert_W^2+\dfrac{\varepsilon t_n}{2}\Vert  g\Vert_W^2 ] \exp \big(C\Vert g\Vert_Wt_n\big).
\end{align}
The  inequality \eqref{estimate-W_T} implies the existence of $t^*>0$ independent of $n$ such that  $\varphi_n \in C([0,t^*],W_n)$ and $\varphi_n$ solves \eqref{Galerkin}. By using \eqref{estimate-W_T} again we deduce that the solution of \eqref{Galerkin} is global in time \textit{i.e.} $\varphi_n \in C([0,T],W_n)$ and 
\begin{align}\label{bound-W}
  \sup_{t\in[0,T]}  \Vert \varphi_n(t)\Vert_W^2+\int_0^{T}[\dfrac{\varepsilon}{2}\Vert  \Delta \varphi_n\Vert_2^2+\varepsilon\Vert \nabla \varphi_n\Vert_2^2+\dfrac{\varepsilon}{2}\Vert \nabla(\varphi_n-\Delta\varphi_n)\Vert_2^2]ds\notag\\ \leq  [\Vert \varphi_0\Vert_W^2+\dfrac{\varepsilon T}{2}\Vert  g\Vert_W^2 ] e^{C\Vert g\Vert_WT}.:=C(T,g,\varphi_0).
\end{align}
\subsubsection*{III. Uniform estimates of $(\partial_t\varphi_n)_n$ with respect to $n$}
We recall that the following Lions-Guelfand triple holds $V\hookrightarrow L^2(D) \hookrightarrow V^\prime.$ Thus, we have
\begin{align}\label{Lions-Guel}
    \forall u\in L^2(D): \quad  \langle u,v\rangle_{V^\prime,V}=(u,v), \quad \forall v\in V.
\end{align}
By construction, note that  $\partial_t(\varphi_n-\Delta\varphi_n) \in W$. Let $v\in V$ and note that
\begin{align*}
   \langle  \partial_t(\varphi_n-\Delta\varphi_n),v\rangle_{V^\prime,V}&= (\partial_t (\varphi_n-\Delta\varphi_n), v)=(\partial_t \varphi_n, v)_V=(P_n\partial_t \varphi_n, v)_V=(\partial_t \varphi_n, P_nv)_V,
\end{align*}
where we used the definition of  $P_n$. Since $P_nv\in W_n$, we can use  \eqref{Galerkin} to get
\begin{align*}
    &  \langle \partial_t(\varphi_n-\Delta\varphi_n),v\rangle_{V^\prime,V}=(\partial_t \varphi_n, P_nv)_V\\&= -(\nabla^\perp\varphi_n\cdot\nabla(g-\Delta \varphi_n)-\varepsilon\Delta(\varphi_n-\Delta\varphi_n+g),P_nv)\\  &=\int_D\nabla^\perp\varphi_n\cdot \nabla \Delta \varphi_n P_nv  dx-\int_D\nabla^\perp\varphi_n\cdot \nabla gP_nv dx-\varepsilon\int_D \nabla(\varphi_n-\Delta\varphi_n+g)\cdot\nabla P_nvdx.
\end{align*}
Therefore, we get
\begin{align*}
       \vert \langle \partial_t(\varphi_n-\Delta\varphi_n),v\rangle_{V^\prime,V}\vert&\leq \Vert\nabla^\perp\varphi_n\Vert_4\Vert \nabla \Delta \varphi_n\Vert_2 \Vert P_nv \Vert_4+\Vert \nabla^\perp\varphi_n\Vert_4\Vert\nabla g\Vert_2\Vert P_nv \Vert_4\\&\quad+\varepsilon\Vert \nabla(\varphi_n-\Delta\varphi_n+g)\Vert_2 \Vert P_nv\Vert_V
\end{align*}
 Now, by using that $P_n$ is an orthogonal projection on $V$, we deduce
\begin{align}\label{est-time-deriv}
  \notag    \vert &\langle \partial_t(\varphi_n-\Delta\varphi_n),v\rangle_{V^\prime,V}\vert\\&\leq M\big(\Vert\varphi_n\Vert_W\Vert \nabla \Delta \varphi_n\Vert_2 +\Vert \varphi_n\Vert_W\Vert\nabla g\Vert_2+\varepsilon\Vert \nabla(\varphi_n-\Delta\varphi_n+g)\Vert_2\big)\Vert v \Vert_V
\end{align}
where $M>0$ depending only on the embeddings $W\hookrightarrow W^{1,4}(D)$ and $V\hookrightarrow L^4(D).$ Consequently, 
\eqref{est-time-deriv} together with  \eqref{bound-W} ensure the boundedness of $(\partial_t(\varphi_n-\Delta\varphi_n))_n$
in $L^2(0,T;V^\prime).$ Thus, the  boundedness of $(\partial_t\varphi_n)_n$
in $L^2(0,T;V).$
\subsubsection*{IV. Compactness and passage to the limit as $n\to +\infty$}
In the previous subsection we proved 
\begin{align}
&(\varphi_n)_n \text{ is bounded in }
 C([0,T],W),\quad (\partial_t\varphi_n)_n \text{ is bounded in }
 L^2(0,T;V),\label{compactness-est-1}\\
 &\nabla(\varphi_n-\Delta\varphi_n) \text{ is bounded in } L^2([0,T]\times D).\label{compactness-est-2}
\end{align}
We recall that $W \hookrightarrow_c W^{1,q}(D)$ for any $q<+\infty$, see \textit{e.g.} \cite[Thm. 1.20]{Roubivcek2005nonlinear}. We introduce the set
$$ B_K=\{f\in L^\infty(0,T;W); \partial_t f\in  L^2(0,T;V):  \Vert f\Vert_{L^\infty(0,T;W)}+\Vert \partial_t f\Vert_{L^2(0,T;V)} \leq K \}, \quad K>0.$$
Note that 
\begin{align}\label{strongcv}
B_K  \text{ is relatively compact in } C([0,T],W^{1,q}(D)) \text{ thanks to \cite[Cor. 4 ]{Simoncompact}}.
\end{align} Consequently, we get the existence of subsequence of $(\varphi_n)_n$ denoted by the same way  and $\varphi \in L^\infty(0,T;W)$ such that the following convergences hold
\begin{align}
    \varphi_n &\buildrel\ast\over\rightharpoonup \varphi \text{ in } L^\infty(0,T;W), \label{cv1}\\
    \partial_t \varphi_n &\rightharpoonup \partial_t\varphi \text{ in } L^2(0,T;V), \label{cv2}\\
     \varphi_n & \to \varphi \text{ in } C([0,T];W^{1,q}(D)),\quad  \forall q<+\infty, \label{cv3}\\
 \nabla(\varphi_n-\Delta\varphi_n) &\rightharpoonup \nabla(\varphi-\Delta\varphi) \text{ in } L^2([0,T]\times D) \text{ and }  \Delta\varphi_n \to \Delta\varphi  \text{ in } L^2([0,T]\times D), \label{cv5}\\
 P_n\varphi_0 &\to \varphi_0 \text{ in } W, \label{cv6}\\
  \text{For any } t\in[0,T]:&\quad    \varphi_n(t)  \rightharpoonup\varphi(t) \text{ in } W. \label{cv4}
\end{align}
Indeed, \eqref{cv1} and \eqref{cv2} are consequence of Banach–Alaoglu theorem, \eqref{cv3} is consequence of \eqref{strongcv}. On the other hand, the first part of \eqref{cv5} follows from Banach–Alaoglu theorem and \eqref{compactness-est-2} while the second part is a consequence of \eqref{compactness-est-1}, \eqref{compactness-est-2} and \cite[Cor. 4]{Simoncompact}. Moreover,
\eqref{cv6} follows by the properties of the projection operator.  Finally,  the first part of 
 \eqref{compactness-est-1} also ensures the boundedness of $(\varphi_n(t))_t$ in $W$ for any $t\in [0,T],$ which ensures the convergence to some limit $\xi_t$ in first step by Banach–Alaoglu theorem and then using \eqref{cv3}, we obtain \eqref{cv4}.\\

By using \eqref{cv1}-\eqref{cv4}, it is possible to pass to the limit in \eqref{Galerkin} as $n\to +\infty$ and obtain
\begin{align}\label{limlit1}
 &       \langle \partial_t(\varphi-\Delta\varphi),e_i\rangle_{V^\prime,V}-\int_D (g-\Delta \varphi)\nabla^\perp\varphi \cdot\nabla e_i dx+\varepsilon\int_D\nabla(\varphi-\Delta\varphi+g)\cdot \nabla e_i dx=0
    \end{align}
    for any $i\in \mathbb{N}$. Moreover, $\varphi_n(0)=P_n\varphi_0  \rightharpoonup \varphi(0)$ thanks to \eqref{cv4} thus \eqref{cv6} ensure that $\varphi$ satisfies the initial condition.
    Then,  by density argument we deduce the existence of $\varphi$ satisfying
    \begin{itemize}
        \item $\varphi \in  C([0,T];W)\text{ and }  \varphi \in L^2(0,T;Y). $ Indeed,  \textit{ a priori} we have
    $$\varphi \in L^\infty(0,T;W)\cap C([0,T];W-w)\cap C([0,T];W^{1,q}(D)), \forall q<+\infty \text{ and } \varphi \in L^2(0,T;Y). $$  
   Now, a classical results based on Lions-Gelfand triple  combined with \cite[Thm. 9.25]{BrezisBook} ensure the claim.
   \item $\varphi$ solves \eqref{HM-eqn-modified} in the following sense
      \begin{align}\label{limlit1-V}
 &       \langle \partial_t(\varphi-\Delta\varphi),\psi\rangle_{V^\prime,V}+\int_D[\varepsilon\nabla(\varphi-\Delta\varphi+g)\cdot \nabla \psi -(g-\Delta \varphi)\nabla^\perp\varphi\cdot \nabla \psi]dx=0, \quad \forall \psi \in V.
    \end{align}  
     \end{itemize}
    \subsubsection*{V. Stability and uniqueness}
     Let $\varphi_1$ and $\varphi_2$ be two solution to \eqref{HM-eqn-modified} corresponding to $(g_1,\varphi_0^1)$ and $(g_2,\varphi_0^2)$, respectively. satisfying \eqref{limlit1-V} and denote by $\varphi=\varphi_1-\varphi_2$. Then,  for any $\psi \in V$ we have 
    \begin{align}\label{diff-2-sols}
        &       \langle \partial_t(\varphi-\Delta\varphi),\psi\rangle_{V^\prime,V}+\int_D[\varepsilon\nabla(\varphi-\Delta\varphi+g)\cdot \nabla \psi  dx \\& \quad-\int_D(g-\Delta \varphi)\nabla^\perp\varphi_1\cdot \nabla \psi dx-\int_D(g_2-\Delta \varphi_2)\nabla^\perp\varphi\cdot \nabla \psi]dx=0.\notag
    \end{align}
    where  $g=g_1-g_2$. Let us start with stability with respect to $V$-norm. For that, 
  set $\psi=\varphi$ in \eqref{diff-2-sols}, we get
 \begin{align*}
        &       \langle \partial_t(\varphi-\Delta\varphi),\varphi\rangle_{V^\prime,V}+\int_D[\varepsilon\nabla(\varphi-\Delta\varphi+g)\cdot \nabla \varphi  dx \\& \quad-\int_D(g-\Delta \varphi)\nabla^\perp\varphi_1\cdot \nabla \varphi dx-\int_D(g_2-\Delta \varphi_2)\overbrace{\nabla^\perp\varphi\cdot \nabla \varphi}^{=0}]dx=0.
 \end{align*}
 Hence
\begin{align*}
        &      \dfrac{1}{2} \dfrac{d}{dt}\Vert \varphi\Vert_V^2+\varepsilon \Vert \nabla \varphi\Vert_2^2+\varepsilon\Vert \Delta \varphi\Vert_2^2  =\varepsilon\int_D[ g \Delta \varphi  dx+\int_D g\nabla^\perp\varphi_1\cdot \nabla \varphi dx -\int_D\Delta \varphi \nabla^\perp\varphi_1\cdot \nabla \varphi dx\\
        &\leq \dfrac{\varepsilon}{2} \Vert g\Vert_2^2+\dfrac{\varepsilon}{2} \Vert \Delta \varphi \Vert_2^2+\Vert g\Vert_4\Vert \nabla^\perp\varphi_1\Vert_4\Vert \nabla \varphi\Vert_2+\dfrac{\varepsilon}{4}\Vert \Delta \varphi\Vert_2^2+\dfrac{1}{\varepsilon}\Vert  \nabla^\perp\varphi_1\cdot \nabla \varphi\Vert_2^2,
         \end{align*}
 which gives 
\begin{align*}
       &      \dfrac{1}{2} \dfrac{d}{dt}\Vert \varphi\Vert_V^2+\varepsilon \Vert \nabla \varphi\Vert_2^2+\dfrac{\varepsilon}{4}\Vert \Delta \varphi\Vert_2^2    \leq  \dfrac{\varepsilon}{2} \Vert g\Vert_2^2+\Vert g\Vert_4\Vert \nabla^\perp\varphi_1\Vert_4\Vert \nabla \varphi\Vert_2+\dfrac{1}{\varepsilon}\Vert   \nabla^\perp\varphi_1\Vert_\infty^2\Vert \nabla \varphi\Vert_2^2\\
       &\leq  \dfrac{\varepsilon}{2} \Vert g\Vert_2^2+\dfrac{1}{2}\Vert g\Vert_4^2+\dfrac{1}{2}\Vert \nabla^\perp\varphi_1\Vert_4^2\Vert \nabla \varphi\Vert_2^2+\dfrac{1}{\varepsilon}\Vert   \nabla^\perp\varphi_1\Vert_\infty^2\Vert \nabla \varphi\Vert_2^2\\
         &\leq  \dfrac{\varepsilon}{2} \Vert g\Vert_2^2+\dfrac{1}{2}\Vert g\Vert_4^2+\dfrac{1}{2}\Vert \varphi_1\Vert_W^2\Vert \nabla \varphi\Vert_2^2+\dfrac{1}{\varepsilon}\Vert   \varphi_1\Vert_Y^2\Vert \nabla \varphi\Vert_2^2\\
         & \leq  \dfrac{\varepsilon}{2} \Vert g\Vert_2^2+\dfrac{1}{2}\Vert g\Vert_4^2+\dfrac{1}{2}\Vert \varphi_1\Vert_W^2\Vert  \varphi\Vert_V^2+\dfrac{1}{\varepsilon}\Vert   \varphi_1\Vert_Y^2\Vert \varphi\Vert_V^2.
\end{align*}
Therefore,  Grönwall's lemma ensures
\begin{align*}
     \sup_{t\in[0,T]}\Vert \varphi(t)\Vert_V^2 \leq [\dfrac{\varepsilon T}{2} \Vert g\Vert_2^2+\dfrac{T}{2}\Vert g\Vert_4^2+\Vert \varphi_0^1-\varphi_0^2\Vert_V^2] \exp{\big(\int_0^T(\dfrac{1}{2}\Vert \varphi_1\Vert_W^2+\dfrac{1}{\varepsilon}\Vert   \varphi_1\Vert_Y^2)ds\big)}.
\end{align*}
Concerning the stability with respect to $W$-norm,    set $\psi=\varphi-\Delta\varphi\in V$ in \eqref{diff-2-sols}, we get
 \begin{align*}
        &   \langle \partial_t(\varphi-\Delta\varphi),\varphi-\Delta\varphi\rangle_{V^\prime,V}+\int_D[\varepsilon\nabla(\varphi-\Delta\varphi+g)\cdot \nabla (\varphi-\Delta\varphi)  dx \\& \quad-\int_D(g-\Delta \varphi)\nabla^\perp\varphi_1\cdot \nabla (\varphi-\Delta\varphi) dx-\int_D(g_2-\Delta \varphi_2)\nabla^\perp\varphi\cdot \nabla (\varphi-\Delta\varphi)]dx=0. 
 \end{align*}
 Note that, by  
 an integration by part formula in time (see \textit{e.g.} \cite[Lem. 5]{Tahraoui2019tools}), one gets
 \begin{align*}
    2 \langle \partial_t(\varphi-\Delta\varphi),\varphi-\Delta\varphi\rangle_{V^\prime,V}=\dfrac{d}{dt}\Vert \varphi-\Delta\varphi\Vert_2^2.
 \end{align*}
 Moreover,  by using Young's inequality we get
 \begin{align*}
     \varepsilon\int_D[\nabla g\cdot \nabla (\varphi-\Delta\varphi)  dx \leq \varepsilon\Vert \nabla g\Vert_2^2+\dfrac{\varepsilon}{4}\Vert \nabla (\varphi-\Delta\varphi)\Vert_2^2.
 \end{align*}
 On the other hand, Young's and Hölder's inequalities ensure
 \begin{align*}
     &\int_D(g-\Delta \varphi)\nabla^\perp\varphi_1\cdot \nabla (\varphi-\Delta\varphi) dx=    \int_Dg\nabla^\perp\varphi_1\cdot \nabla (\varphi-\Delta\varphi) dx-    \int_D\Delta \varphi\nabla^\perp\varphi_1\cdot \nabla \varphi dx\\
     &\leq \dfrac{1}{\varepsilon}\Vert g\Vert_4^2\Vert \nabla^\perp \varphi_1\Vert_4^2+\dfrac{\varepsilon}{4}\Vert \nabla (\varphi-\Delta\varphi)\Vert_2^2+\Vert \Delta \varphi\Vert_2\Vert \nabla^\perp \varphi_1\Vert_4\Vert \nabla \varphi\Vert_4\\
      &\leq \dfrac{C}{\varepsilon}\Vert g\Vert_4^2\Vert  \varphi_1\Vert_W^2+\dfrac{\varepsilon}{4}\Vert \nabla (\varphi-\Delta\varphi)\Vert_2^2+C\Vert  \varphi\Vert_W^2\Vert  \varphi_1\Vert_W,
 \end{align*}
 where $C>0$ depends only on the embedding $W \hookrightarrow W^{1,4}(D).$ Next, we have
 \begin{align*}
    & \int_D(g_2-\Delta \varphi_2)\nabla^\perp\varphi\cdot \nabla (\varphi-\Delta\varphi)]dx\leq \Vert g_2-\Delta \varphi_2\Vert_4\Vert \nabla^\perp\varphi\Vert_4\Vert \nabla (\varphi-\Delta\varphi)\Vert_2\\
     &\leq \dfrac{\varepsilon}{4}\Vert \nabla (\varphi-\Delta\varphi)\Vert_2^2+\dfrac{1}{\varepsilon}\Vert g_2-\Delta \varphi_2\Vert_4^2\Vert \nabla^\perp\varphi\Vert_4^2\\
       &\leq \dfrac{\varepsilon}{4}\Vert \nabla (\varphi-\Delta\varphi)\Vert_2^2+\dfrac{1}{\varepsilon}[2\Vert g_2\Vert_4^2+2C_*\Vert \varphi_2\Vert_Y^2]\Vert \varphi\Vert_W^2.
 \end{align*}
Therefore we get
\begin{align*}
    \dfrac{d}{dt}\Vert \varphi-\Delta\varphi\Vert_2^2+\dfrac{\varepsilon}{4}\Vert \nabla (\varphi-\Delta\varphi)\Vert_2^2 \leq  \varepsilon\Vert \nabla g\Vert_2^2+ \dfrac{C}{\varepsilon}\Vert g\Vert_4^2\Vert  \varphi_1\Vert_W^2+\dfrac{1}{\varepsilon}[2\Vert g_2\Vert_4^2+2C\Vert \varphi_2\Vert_Y^2+C\Vert  \varphi_1\Vert_W]\Vert \varphi\Vert_W^2.
\end{align*}
In conclusion, there exists $\mathbf{C}>0$ depends only on Sobolev embeddings such that
\begin{align*}
    \dfrac{d}{dt}\Vert \varphi-\Delta\varphi\Vert_2^2+\dfrac{\varepsilon}{4}\Vert \nabla (\varphi-\Delta\varphi)\Vert_2^2 \leq  \varepsilon\Vert \nabla g\Vert_2^2+ \dfrac{\mathbf{C}}{\varepsilon}\Vert g\Vert_4^2\Vert  \varphi_1\Vert_W^2+\dfrac{\mathbf{C}}{\varepsilon}[\Vert g_2\Vert_4^2+\Vert \varphi_2\Vert_Y^2+\Vert  \varphi_1\Vert_W]\Vert \varphi\Vert_W^2.
\end{align*}
Finally,  by using the fact that $\Vert \varphi \Vert_W^2 \leq C_W\Vert \varphi-\Delta\varphi\Vert_2^2,$ where $C_W>0.$ Grönwall's lemma ensures
\begin{align*}
     &\sup_{t\in[0,T]}\Vert \varphi(t)\Vert_W^2+\varepsilon\int_0^T\Vert \nabla (\varphi-\Delta\varphi)\Vert_2^2 ds\\&\leq \overline{\mathbf{C}}[\Vert \varphi_0^1-\varphi_0^2\Vert_W^2+\varepsilon T\Vert \nabla g\Vert_2^2+ \dfrac{1}{\varepsilon}\Vert g\Vert_4^2\int_0^T\Vert  \varphi_1\Vert_W^2 ds] \exp{\big(\dfrac{\mathbf{C}}{\varepsilon}\int_0^T(\Vert g_2\Vert_4^2+\Vert \varphi_2\Vert_Y^2+\Vert  \varphi_1\Vert_W)ds\big)}.
\end{align*}
where $\overline{\mathbf{C}}>0$ independent of $g_i, \varepsilon, \varphi_i; i=1,2.$
%%%%%$\big(W^{2,p}(D)\big)_{2\leq p<+\infty}$%%%
 \section{Hasegawa-Mima equation  ($p$-integrable density, $2\leq p<+\infty$) }\label{section-exis-uniq-Yod-1}
We start this section by proving the following approximation lemma, which serves in the construction of $W^{2,p}(D)\cap H^1_0(D)$ solution to Hasegawa-Mima equation \eqref{HM-eqn}. 
\begin{lemma}\label{lemma-approximation}
Let $2\leq p <+\infty.$
  For any $h\in L^p(D)$, there exists $(h_\delta)_{\delta>0}  \subseteq  W$ such that   $h_\delta \to h$ in $L^p(D)$ as $\delta \to 0.$ Moreover,  the following holds
   \begin{align}\label{properties-appro}
  \text{ for any } \delta >0 : \quad \Vert  h_\delta\Vert_p^p\leq \Vert h\Vert_p^p  \text{ and }      \Vert h_\delta\Vert_2^2+\delta\int_D \vert \nabla h_\delta\vert^2dx \leq  \Vert h\Vert_2^2.    
    \end{align}
    Moreover, if $h\in L^\infty(D)$ then $h_\delta \in L^\infty(D)$ and the following inequality holds
    \begin{align}\label{Linfty-bound-appro}
        \Vert h_\delta \Vert_\infty \leq  \Vert h \Vert_\infty, \quad \forall \delta >0.
    \end{align}
\end{lemma}
\begin{proof}
Let $h\in L^p(D)$ and $\delta>0$, by classical arguments (Lax-Milgram theorem and elliptic regularity theory, see  \textit{e.g.} \cite[Thm. 9.25]{BrezisBook}) there exists a unique $h_\delta \in H^2(D)\cap H^1_0(D)=W$ such that 
    \begin{align}\label{form-Lax-Mil}
     \int_D h_\delta\psi dx+\delta\int_D \nabla h_\delta\cdot \nabla \psi dx=\int_D h\psi dx ,  \quad \forall \psi \in H^1_0(D).  
    \end{align}
    Set $\psi=h_\delta$ in \eqref{form-Lax-Mil} and use Cauchy–Schwarz and Young inequalities   to get 
    \begin{align*}
          \int_D h_\delta^2 dx+\delta\int_D \vert \nabla h_\delta\vert^2dx=\int_D hh_\delta dx \leq \Vert h\Vert_2\Vert h_\delta\Vert_2 \leq \dfrac{1}{2} \Vert h\Vert_2^2+\dfrac{1}{2}\Vert h_\delta\Vert_2^2
    \end{align*}
    which gives 
    \begin{align}\label{est-delta}
        \dfrac{1}{2}\Vert h_\delta\Vert_2^2+\delta\int_D \vert \nabla h_\delta\vert^2dx \leq \dfrac{1}{2} \Vert h\Vert_2^2.    
    \end{align}
     Now, we set $\psi= \vert T_k(h_\delta)\vert^{p-2}T_k(h_\delta) \in H^1_0(D)$ in \eqref{form-Lax-Mil} where $T_k(h_\delta)=\max(\min(h_\delta,k),-k)$ for any $k\in \mathbb{R}^+.$ Therefore
                \begin{align*}
          &\int_D h_\delta \vert T_k(h_\delta)\vert^{p-2}T_k(h_\delta) dx+(p-1)\delta\int_D \vert \nabla h_\delta\vert^2T_k^\prime(h_\delta) \vert T_k(h_\delta)\vert^{p-2} dx\\
          &=\int_D h \vert T_k(h_\delta)\vert^{p-2}T_k(h_\delta) dx \leq \dfrac{1}{p}\Vert h\Vert_p^p+\dfrac{p-1}{p} \Vert  T_k(h_\delta)\Vert_p^p\leq \dfrac{1}{p}\Vert h\Vert_p^p+\dfrac{p-1}{p} \Vert  h_\delta\Vert_p^p.
       \end{align*}
      Since the integrand in the first line of the above inequality is non-negative, we can use Fatou's lemma and pass to the limit as $k\to +\infty$ to get 
       \begin{align*}
          &\int_D  \vert h_\delta)\vert^{p} dx+(p-1)\delta\int_D \vert \nabla h_\delta\vert^2 \vert h_\delta\vert^{p-2} dx\leq \dfrac{1}{p}\Vert h\Vert_p^p+\dfrac{p-1}{p} \Vert  h_\delta\Vert_p^p.
       \end{align*}
       Thus $
           \dfrac{1}{p} \Vert  h_\delta\Vert_p^p+(p-1)\delta\int_D \vert \nabla h_\delta\vert^2 \vert h_\delta\vert^{p-2} dx\leq \dfrac{1}{p}\Vert h\Vert_p^p$. In particular, we have 
           \begin{align}\label{unifor-convex}  \text{ for any } \delta >0 : \quad \Vert  h_\delta\Vert_p^p\leq \Vert h\Vert_p^p  . 
           \end{align}
           It follows from \eqref{form-Lax-Mil},  \eqref{est-delta} and the density of $H^1_0(D)$ in $L^{p/p-1}(D)$  that $h_\delta \rightharpoonup h$ in $L^p(D)$ as $\delta \to 0.$ Now, \eqref{unifor-convex} ensures that $\displaystyle\limsup_{\delta \to 0} \Vert  h_\delta\Vert_p^p\leq \Vert h\Vert_p^p$ and thus $\Vert h_\delta - h\Vert_p \to 0$ as $\delta \to 0$ by uniform convexity argument, see \textit{e.g.} \cite[Prop. 3.32]{BrezisBook}. Finally, \eqref{Linfty-bound-appro} is a consequence of  the maximum principle (see \textit{e.g.} \cite[Thm. 9.27]{BrezisBook}).
\end{proof}
Now, let us prove \autoref{TH-existence-HM-0}.
    \subsection{Proof of \autoref{TH-existence-HM-0} (the existence)}\label{SubSection-proof1}
    Let $2\leq p< +\infty$,  $g:=\log (n_0) \in L^p(D)$ and      $\varphi_0-\Delta \varphi_0\in L^p(D)$ such that $\varphi_0=0$ on $\partial D.$ \\
    Thanks to \autoref{lemma-approximation}, there exists a sequence  $(g_\varepsilon)_\varepsilon \subseteq W$ such that $\displaystyle\lim_{\varepsilon \to 0 }g_\varepsilon=g$ in $L^p(D)$ and such that  \eqref{properties-appro} holds. On the other hand, by density  argument there exists a sequence $(\varphi_\varepsilon^0)_\varepsilon \subseteq Y$ such that 
    \begin{align}\label{approx-initial-data}
     \varphi_\varepsilon^0 \to \varphi_0  \text{ in }  H^1_0(D)\cap W^{2,p}(D)  \text{ as } \epsilon \to 0.
    \end{align}

Let $\epsilon>0$, by using \autoref{THM1}, there exists  a unique   $ \varphi_\epsilon \in C([0,T];W) \cap L^2(0,T;Y)  $ such that for any $\psi \in V$ we have
     \begin{align}\label{DEF-form-epsilon}
 &       \langle \partial_t(\varphi_\epsilon-\Delta\varphi_\epsilon),\psi\rangle_{V^\prime,V}+\int_D[\varepsilon\nabla(\varphi_\epsilon-\Delta\varphi_\epsilon+g_\epsilon)\cdot \nabla \psi -(g_\epsilon-\Delta \varphi_\epsilon)\nabla^\perp\varphi_\epsilon\cdot \nabla \psi]dx=0.
    \end{align}
\subsubsection*{I. Uniform estimate w.r.t. $\varepsilon$}
Let $M\in \mathbb{R}^+$ and denote by $\beta_M$ the even $C^2$-function defined on $\lambda\in [0,+\infty[$ by
\begin{align}\label{aprox-norm-p}
    \beta_M(\lambda)=\lambda^p1_{[0,M[}(\lambda)+\big( M^p+\dfrac{p(p-1)}{2}M^{p-2}(\lambda-M)^2+pM^{p-1}(\lambda-M)\big)1_{[M,+\infty[}(\lambda).
\end{align}
This is a non decreasing sequence of non-negative functions such that $\displaystyle\lim_M\beta_M(\lambda)=\beta (\lambda)=\vert \lambda\vert^p$ point-wise. Moreover, it is convex function \textit{i.e.}  $\beta_M^{\prime\prime}\geq 0.$\\

Note that $\psi=\beta_M^\prime(\varphi_\epsilon-\Delta\varphi_\epsilon+g_\epsilon) \in V$ and by using \eqref{DEF-form-epsilon}, we get
   \begin{align*}
 &       \langle \partial_t(\varphi_\epsilon-\Delta\varphi_\epsilon),\beta_M^\prime(\varphi_\epsilon-\Delta\varphi_\epsilon+g_\epsilon)\rangle_{V^\prime,V}+\int_D\varepsilon\overbrace {\vert\nabla(\varphi_\epsilon-\Delta\varphi_\epsilon+g_\epsilon)\vert^2 \beta_M ^{\prime\prime}(\varphi_\epsilon-\Delta\varphi_\epsilon+g_\epsilon)}^{\geq 0}dx\\
 &-\int_D(g_\epsilon-\Delta \varphi_\epsilon)\nabla^\perp\varphi_\epsilon\cdot \nabla \beta_M^\prime(\varphi_\epsilon-\Delta\varphi_\epsilon+g_\epsilon)dx=0.
 \end{align*}
 By using an integration by part formula in time (see \textit{e.g.} \cite[Lem. 5]{Tahraoui2019tools}), we get
 \begin{align*}
     \langle \partial_t(\varphi_\epsilon-\Delta\varphi_\epsilon),\beta_M^\prime(\varphi_\epsilon-\Delta\varphi_\epsilon+g_\epsilon)\rangle_{V^\prime,V}&=\langle \partial_t(\varphi_\epsilon-\Delta\varphi_\epsilon+g_\epsilon),\beta_M^\prime(\varphi_\epsilon-\Delta\varphi_\epsilon+g_\epsilon)\rangle_{V^\prime,V}\\
     &=\dfrac{d}{dt}\int_D \beta_M(\varphi_\epsilon-\Delta\varphi_\epsilon+g_\epsilon) dx.
 \end{align*}
 On the other hand, by using that $g_\epsilon-\Delta \varphi_\epsilon=0$ on $\partial D$ and  $\nabla^\perp\varphi_\epsilon \cdot \nabla \varphi_\varepsilon=0$, we get 
 \begin{align*}
     \int_D(g_\epsilon-\Delta \varphi_\epsilon)\nabla^\perp\varphi_\epsilon\cdot \nabla \beta_M^\prime(\varphi_\epsilon-\Delta\varphi_\epsilon+g_\epsilon)dx&=-     \int_D\nabla^\perp\varphi_\epsilon \cdot \nabla(g_\epsilon+\varphi_\varepsilon-\Delta \varphi_\epsilon)\beta_M^\prime(\varphi_\epsilon-\Delta\varphi_\epsilon+g_\epsilon)dx\\
    & =-     \int_D\nabla^\perp\varphi_\epsilon \cdot \nabla\beta_M(\varphi_\epsilon-\Delta\varphi_\epsilon+g_\epsilon)dx\\
   & =-     \int_{\partial D}\nabla^\perp\varphi_\epsilon \cdot  \eta \beta_M(\varphi_\epsilon-\Delta\varphi_\epsilon+g_\epsilon)d\sigma=0,
 \end{align*}
 since $\beta_M(\varphi_\epsilon-\Delta\varphi_\epsilon+g_\epsilon)=0$ on $\partial D.$ Therefore we get  \begin{align*}
\int_D \beta_M(\varphi_\epsilon(t)-\Delta\varphi_\epsilon(t)+g_\epsilon) dx&+\varepsilon\int_0^t\int_D\vert\nabla(\varphi_\epsilon-\Delta\varphi_\epsilon+g_\epsilon)\vert^2 \beta_M ^{\prime\prime}(\varphi_\epsilon-\Delta\varphi_\epsilon+g_\epsilon)dx ds\\ &= \int_D \beta_M(\varphi_\epsilon^0-\Delta\varphi_\epsilon^0+g_\epsilon) dx.
 \end{align*}
 Now, we use  the monotone convergence theorem, as $M\to +\infty$ to deduce for any $t\in [0,T]$
 \begin{align}
\int_D \vert\varphi_\epsilon(t)-\Delta\varphi_\epsilon(t)+g_\epsilon\vert^p dx&+\varepsilon p(p-1)\int_0^t\int_D\vert\nabla(\varphi_\epsilon-\Delta\varphi_\epsilon+g_\epsilon)\vert^2 \vert \varphi_\epsilon-\Delta\varphi_\epsilon+g_\epsilon\vert^{p-2}dx ds \notag \\&\label{inequality-uniq-Yodu}=\int_D \vert \varphi_\epsilon^0-\Delta\varphi_\epsilon^0+g_\epsilon\vert^p dx 
\leq  2^{p-1}(\Vert \varphi_\epsilon^0-\Delta\varphi_\epsilon^0\Vert_p^p+\Vert g_\epsilon\Vert_p^p).
 \end{align}
 Thus, for any $t\in [0,T]$, we have
 \begin{align*}
\int_D \vert\varphi_\epsilon(t)-\Delta\varphi_\epsilon(t)\vert^p dx &\leq 2^{p-1}(\int_D \vert\varphi_\epsilon(t)-\Delta\varphi_\epsilon(t)+g_\epsilon\vert^p dx +\Vert g_\epsilon\Vert_p^p)\\
&\leq  4^{p-1}(\Vert \varphi_\epsilon^0-\Delta\varphi_\epsilon^0\Vert_p^p+2\Vert g_\epsilon\Vert_p^p)\\
&\leq  4^{p-1}(\Vert \varphi_0-\Delta\varphi_0\Vert_p^p+1+2\Vert g\Vert_p^p)
 \end{align*}
 by using \eqref{properties-appro} and the convergence of $\varphi_\varepsilon^0 \to \varphi_0$ in $W^{2,p}(D)$. Consequently, we proved
 \begin{align}\label{bound-p-pot}
     (\varphi_\epsilon)_\varepsilon \text{ is bounded uniformly w.r.t. } \varepsilon  \text{ in } C([0,T],V\cap W^{2,p}(D)).
 \end{align}
 \subsubsection*{II. Uniform estimates of $(\partial_t\varphi_\varepsilon)_\varepsilon$ with respect to $\varepsilon$}
 
We recall    $W\hookrightarrow V\hookrightarrow L^2(D) \hookrightarrow V^\prime\hookrightarrow W^\prime.$ Thus,
$\partial_t(\varphi_\varepsilon-\Delta\varphi_\varepsilon) \in L^2(0,T;W^\prime).$ Now, let $v\in W$ and note that 
\begin{align*}
   \langle  \partial_t(\varphi_\varepsilon-\Delta\varphi_\varepsilon),v\rangle_{W^\prime,W}&:=   \langle  \partial_t(\varphi_\varepsilon-\Delta\varphi_\varepsilon),v\rangle_{V^\prime,V}.
\end{align*}
By using \eqref{DEF-form-epsilon} and since $\varphi_\varepsilon \in L^2(0,T;Y)$, we get
\begin{align*}
    &  \langle  \partial_t(\varphi_\varepsilon-\Delta\varphi_\varepsilon),v\rangle_{W^\prime,W}=-\int_D\nabla^\perp\varphi_\varepsilon\cdot \nabla  v  \Delta \varphi_\varepsilon dx+\int_D\nabla^\perp\varphi_\varepsilon\cdot \nabla  v g_\varepsilon  dx+\varepsilon\int_D (\varphi_\varepsilon-\Delta\varphi_\varepsilon+g_\varepsilon)\Delta v  dx.
\end{align*}
Therefore,  Hölder's inequality ensures
\begin{align*}
       \vert \langle \partial_t(\varphi_\varepsilon-\Delta\varphi_\varepsilon),v\rangle_{W^\prime,W}\vert&\leq \Vert\nabla^\perp\varphi_\varepsilon\Vert_4\Vert  \Delta \varphi_\varepsilon\Vert_2 \Vert \nabla v \Vert_4+\Vert \nabla^\perp\varphi_\varepsilon\Vert_4\Vert g_\varepsilon\Vert_2\Vert \nabla v \Vert_4\\&\quad+\varepsilon\Vert \varphi_\varepsilon-\Delta\varphi_\varepsilon+g_\varepsilon\Vert_2 \Vert \Delta v\Vert_2.
\end{align*}
Consequently, by using \eqref{properties-appro}
 \begin{align}\label{est-time-deriv-epsilon}
 \notag   \vert \langle \partial_t(\varphi_\varepsilon-\Delta\varphi_\varepsilon),v\rangle_{W^\prime,W}\vert &\leq M\big(\Vert\varphi_\varepsilon\Vert_W^2+\Vert \varphi_\varepsilon\Vert_W\Vert g_\varepsilon\Vert_2+\varepsilon\Vert \varphi_\varepsilon-\Delta\varphi_\varepsilon\Vert_2+ \varepsilon\Vert g_\varepsilon\Vert_2\big)\Vert v \Vert_W\\
    &\leq M\big(\Vert\varphi_\varepsilon\Vert_W^2+\Vert \varphi_\varepsilon\Vert_W\Vert g\Vert_2+\varepsilon\Vert \varphi_\varepsilon\Vert_W+ \varepsilon\Vert g\Vert_2\big)\Vert v \Vert_W
\end{align}
where $M>0$ depending only on the embeddings $W\hookrightarrow W^{1,4}(D)$. Consequently, 
\eqref{est-time-deriv-epsilon} together with  \eqref{bound-p-pot} ensure the boundedness of $(\partial_t(\varphi_\varepsilon-\Delta\varphi_\varepsilon))_\varepsilon$
in $L^\infty(0,T;W^\prime).$ Thus, the  boundedness of $(\partial_t\varphi_\varepsilon)_\varepsilon$
in $L^\infty(0,T;L^2(D)).$ 
\subsubsection*{III. Compactness and passage to the limit as $\varepsilon\to 0$}
In the previous subsection we proved 
\begin{align}
 (\varphi_\varepsilon)_\varepsilon&\text{ is bounded in }    C([0,T],V\cap W^{2,p}(D)),\quad (\partial_t\varphi_\varepsilon)_\varepsilon \text{ is bounded in }
 L^\infty(0,T;L^2(D)).\label{compactness-est-1_epsilon}\end{align}
We recall that $V\cap W^{2,p}(D)\hookrightarrow W \hookrightarrow_c  W^{1,q}(D)$ for any $q<+\infty$, see \textit{e.g.} \cite[Thm. 1.20]{Roubivcek2005nonlinear}. We introduce the set
$$ \overline{B}_K=\{f\in L^\infty(0,T;W); \partial_t f\in  L^\infty(0,T;L^2(D)):  \Vert f\Vert_{L^\infty(0,T;W)}+\Vert \partial_t f\Vert_{L^\infty(0,T;L^2(D))} \leq K \}, \quad K>0.$$
Note that 
\begin{align}\label{strongcv-eps}
\overline{B}_K  \text{ is relatively compact in } C([0,T],W^{1,q}(D)) \text{ thanks to \cite[Cor. 4 ]{Simoncompact}}.
\end{align} Consequently, we get the existence of subsequence of $(\varphi_\varepsilon)_\varepsilon$ denoted by the same way  and $\varphi \in L^\infty(0,T;V\cap W^{2,p}(D))$ such that the following convergences, as $\varepsilon \to 0,$ hold
\begin{align}
    \varphi_\varepsilon &\buildrel\ast\over\rightharpoonup \varphi \text{ in } L^\infty(0,T;V\cap W^{2,p}(D)), \label{cv1-eps}\\
    \partial_t \varphi_\varepsilon &\buildrel\ast\over\rightharpoonup \partial_t\varphi \text{ in } L^\infty(0,T;L^2(D)), \label{cv2-eps}\\
     \varphi_\varepsilon & \to \varphi \text{ in } C([0,T];W^{1,q}(D)),\quad  \forall q<+\infty, \label{cv3-eps}.
\end{align}
Indeed, \eqref{cv1-eps} and \eqref{cv2-eps} are consequence of Banach–Alaoglu theorem, \eqref{cv3-eps} is consequence of \eqref{strongcv-eps}. On the other hand, recall that \autoref{lemma-approximation} ensures
\begin{align}
    g_\varepsilon \to g \text{ in }  L^p(D) \label{cv4-eps}.
\end{align}
Moreover, \eqref{approx-initial-data} and \eqref{cv3-eps}
ensures that 
\begin{align}
\varphi_\varepsilon(0)=\varphi_\varepsilon^0 \text{ converges  in } W^{1,q}(D) \text{ to } \varphi(0)=\varphi_0 \in V\cap W^{2,p}(D).     
\end{align}
Furthermore,  \eqref{compactness-est-1_epsilon} ensures that $\varphi_\varepsilon(t) \rightharpoonup \varphi(t)$ in  $V\cap W^{2,p}(D)$ for any $t\in [0,T]$ by an argument similar to that used to obtain \eqref{cv4}.\\

Let $\psi \in W$ and recall that from \eqref{DEF-form-epsilon} we have
\begin{align}\label{DEF-form-epsilon--}
 &       \langle \partial_t(\varphi_\epsilon-\Delta\varphi_\epsilon),\psi\rangle_{W^\prime,W}-\int_D (g_\epsilon-\Delta \varphi_\epsilon)\nabla^\perp\varphi_\epsilon\cdot \nabla \psi dx=\varepsilon\int_D(\varphi_\epsilon-\Delta\varphi_\epsilon+g_\epsilon)\Delta \psi.
    \end{align}
    Thanks to \autoref{lemma-approximation} and 
  \eqref{compactness-est-1_epsilon}, we get 
  \begin{align*}
      \lim_{\varepsilon \to 0} \varepsilon\int_D(\varphi_\epsilon-\Delta\varphi_\epsilon+g_\epsilon)\Delta \psi=0.
  \end{align*}
  Therefore, it is possible to use \eqref{cv1-eps}-\eqref{cv4-eps} and pass to the limit as $\varepsilon \to 0$ in \eqref{DEF-form-epsilon--} to get
\begin{align}\label{DEF-weak-solu}
 &       \langle \partial_t(\varphi-\Delta\varphi),\psi\rangle_{W^\prime,W}-\int_D (g-\Delta \varphi)\nabla^\perp\varphi\cdot \nabla \psi dx=0, \quad \forall \psi \in W.
    \end{align}
  In conclusion, we proved  the existence of $\varphi$ satisfying
    \begin{itemize}
        \item $\varphi \in  C([0,T];W^{1,q}(D)) \text{ for any } q \text{ finite and }  \varphi \in L^\infty(0,T;V\cap W^{2,p}(D)). $
        \item $\varphi \in C([0,T];V\cap W^{2,p}(D)-w)$, see \textit{e.g.}  \cite[Lem 1.4 p. 263]{Temam2024navier}.
   \item $\varphi$ satisfies $\varphi(0)=\varphi_0$ and  solves \eqref{HM-eqn} in the following sense   
      \begin{align}\label{limlit1-V}
 &       \langle \partial_t(\varphi-\Delta\varphi),\psi\rangle_{W^\prime,W}-\int_D(g-\Delta \varphi)\nabla^\perp\varphi\cdot \nabla \psi dx=0, \quad \forall \psi \in W.
    \end{align}  
     \end{itemize}
     We end this section by the following estimate, which play an important role for a uniqueness result. From \eqref{inequality-uniq-Yodu}, we have
     \begin{align*}
\text{ for any } t\in [0,T]: \quad
\int_D \vert\varphi_\epsilon(t)-\Delta\varphi_\epsilon(t)+g_\epsilon\vert^p dx \leq \int_D \vert \varphi_\epsilon^0-\Delta\varphi_\epsilon^0+g_\epsilon\vert^p dx 
 \end{align*}
 By using the lower semicontinuity of the  weak convergence in $V\cap W^{2,p}(D)$, \autoref{lemma-approximation} and  \eqref{approx-initial-data}, we get
     \begin{align}\label{inequality-uniq-Yodu-2}
\text{ for any } t\in [0,T]: \quad
\int_D \vert\varphi(t)-\Delta\varphi(t)+g\vert^p dx \leq \int_D \vert \varphi_0-\Delta\varphi_0+g\vert^p dx 
 \end{align}
 which gives $
     \Vert \varphi(t)-\Delta\varphi(t)\Vert_p \leq  4(\Vert \varphi_0-\Delta\varphi_0\Vert_p+2\Vert g\Vert_p).$
 So that there exists $K_*>0$ independent of $p$ such that
   \begin{align*}
   \sup_{t\in[0,T]}   \Vert \varphi(t)-\Delta\varphi(t)\Vert_p \leq  K_*(\Vert \varphi_0-\Delta\varphi_0\Vert_p+\Vert g\Vert_p).
 \end{align*}

 To study the uniqueness, we need to perform some integration by parts in time. This is the aim of  the next lemma.   \begin{lemma}\label{lem-IPP-time}
 Let $\varphi$ be the solution to \eqref{HM-eqn} given in \autoref{TH-existence-HM-0}. The following equality holds
    \begin{align}\label{IPP-equality}
       2    \langle \partial_t(\varphi-\Delta\varphi),\varphi\rangle_{W^\prime,W}=\dfrac{d}{dt}\Vert \varphi\Vert_V^2 \text{ in the sense of distribution on } ]0,T[.
         \end{align}
     \end{lemma}    
   \begin{proof}
       Although the proof is based on standard argument (see \textit{e.g.} \cite{Temam2024navier}), we include it here for the convenience of the reader.
       Denote by $\widetilde{\varphi}$ the extension to the real line $\mathbb{R}$ of $\varphi$ by $0$ outside the interval $[0,T].$ Next, consider the regularization by convolution, with some given kernel, of $\widetilde{\varphi}$ to obtain a sequence of function $(\varphi_l)_l$ such that: $  \text{for any } l \in \mathbb{N}, \quad \varphi_l \in C^\infty([0,T];W)$ and 
       \begin{align}\label{cv-convolution}
          \varphi_l  \to \varphi \text{ in }  L^2_{loc}(]0,T[;W) \text{ and } \partial_t(\varphi_l-\Delta \varphi_l) \to \partial_t(\varphi-\Delta \varphi) \text{ in }  L^2_{loc}(]0,T[;W^\prime).
       \end{align}
       On the other hand, by using \eqref{Lions-Guel} we have
       \begin{align}
      \label{cv-conv-1}     2 \langle \partial_t(\varphi_l-\Delta\varphi_l),\varphi_l\rangle_{W^\prime,W}= 2\langle \partial_t(\varphi_l-\Delta\varphi_l),\varphi_l\rangle_{V^\prime,V}&=2(\partial_t(\varphi_l-\Delta\varphi_l),\varphi_l)\\
     \text{(after integration by parts with respect to x)}       &=\dfrac{d}{dt}\Vert \varphi_l\Vert_V^2, \quad \forall l \in \mathbb{N}.\notag
       \end{align}
       It follows from \eqref{cv-convolution} that 
       \begin{align*}
        \lim_{l\to +\infty}   \Vert \varphi_l\Vert_V^2 =\Vert \varphi\Vert_V^2  \text{ and } \lim_{l\to +\infty} \langle \partial_t(\varphi_l-\Delta\varphi_l),\varphi_l\rangle_{W^\prime,W}=\langle \partial_t(\varphi-\Delta\varphi),\varphi\rangle_{W^\prime,W} \text{ in } L^1_{loc}(]0,T[).
       \end{align*}
       The last convergence allows to pass to the limit (in  the sense of distribution), as $l\to +\infty$ in \eqref{cv-conv-1}  to conclude.
   \end{proof}
\subsection{Proof of \autoref{THm-uniqu-YuD}  (the uniqueness)}
We prove  the uniqueness in $\bigcap_{1\leq p<+\infty}\big(W^{2,p}(D)\big)$. It is worth mentioning that the proof is followed by an argument similar to \cite{Yudovich1995uniqueness}. 
     Let $\varphi_1$ and $\varphi_2$ be two solutions to \eqref{Def-sol_HM} associated to the same data $(\log(n_0),\varphi_0)$ and denote by $\varphi=\varphi_1-\varphi_2$. Then,  for any $\psi \in W$ we have 
    \begin{align}\label{diff-2-sols-HM}
        &       \langle \partial_t(\varphi-\Delta\varphi),\psi\rangle_{W^\prime,W}-\int_D\log (n_0)\nabla^\perp\varphi\cdot \nabla \psi dx\\& \quad+\int_D\Delta \varphi\nabla^\perp\varphi_1\cdot \nabla \psi dx+\int_D\Delta \varphi_2\nabla^\perp\varphi\cdot \nabla \psi dx=0.\notag
    \end{align} 
  Set $\psi=\varphi$ in \eqref{diff-2-sols-HM} and use that $\nabla^\perp\varphi\cdot \nabla \varphi=0$, we get
 \begin{align*}
     \langle \partial_t(\varphi-\Delta\varphi),\varphi\rangle_{W^\prime,W}+\int_D\Delta \varphi\nabla^\perp\varphi_1\cdot \nabla \varphi dx=0.
 \end{align*}
 Thanks to \eqref{IPP-equality}, we write
$   \dfrac{d}{dt}\Vert \varphi\Vert_V^2+2\int_D\Delta \varphi\nabla^\perp\varphi_1\cdot \nabla \varphi dx=0.
 $ An integration by parts ensures
\begin{align*}
    &\int_D\Delta \varphi\nabla^\perp\varphi_1\cdot \nabla \varphi dx=\sum_{i=1,2}\int_D\partial_i^2 \varphi\nabla^\perp\varphi_1\cdot \nabla \varphi dx\\&=-\sum_{i=1,2}\int_D\partial_i \varphi\nabla^\perp\varphi_1\cdot \nabla \partial_i\varphi dx-\sum_{i=1,2}\int_D\partial_i \varphi\nabla^\perp\partial_i\varphi_1\cdot \nabla \varphi dx+\int_{\partial D}  \nabla\varphi\cdot \eta\nabla^\perp\varphi_1\cdot \nabla \varphi d\sigma,\\
    &=-\dfrac{1}{2}\sum_{i=1,2}\int_D\nabla^\perp\varphi_1\cdot \nabla (\partial_i\varphi)^2 dx-\sum_{i=1,2}\int_D\partial_i \varphi\nabla^\perp\partial_i\varphi_1\cdot \nabla \varphi dx+\int_{\partial D}  \nabla\varphi\cdot \eta\nabla^\perp\varphi_1\cdot \nabla \varphi d\sigma\\
   & =-\dfrac{1}{2}\sum_{i=1,2}\int_{\partial D}\nabla^\perp\varphi_1\cdot \eta (\partial_i\varphi)^2 d\sigma-\sum_{i=1,2}\int_D\partial_i \varphi\nabla^\perp\partial_i\varphi_1\cdot \nabla \varphi dx+\int_{\partial D}  \nabla\varphi\cdot \eta\nabla^\perp\varphi_1\cdot \nabla \varphi d\sigma.
\end{align*}
Note that the boundary terms equals $0$. Indeed,  since $(\varphi_i)_{i=1,2}=0$ on $\partial D$, we have that $(\nabla\varphi_i)_{i=1,2}$ is 
normal to $\partial D$. Since $D\subseteq \mathbb{R}^2,$  the tangent vector  on the boundary  is given by $\tau=(-\eta_2,\eta_1).$  Thus, we have
\begin{align*}
 \text{ for } i=1,2: \quad  (\nabla^\perp \varphi_i\cdot \eta)\tau= (\nabla \varphi_i\cdot \tau)\tau:= \nabla \varphi_i-(\nabla \varphi_i\cdot \eta) \eta=0,
\end{align*}
which implies  that $\int_{\partial D}\nabla^\perp\varphi_1\cdot \eta [(\partial_1\varphi)^2+ (\partial_2\varphi)^2]d\sigma=0.$ For the other boundary term, note that
\begin{align*}
    \int_{\partial D}  \nabla\varphi\cdot \eta\nabla^\perp\varphi_1\cdot \nabla \varphi d\sigma&=\int_{\partial D}  \nabla\varphi\cdot \eta\nabla^\perp\varphi_1\cdot [(\nabla \varphi\cdot \eta)\eta+ \underbrace{(\nabla \varphi\cdot \tau)}_{=0}\tau ]d\sigma\\
    &=\int_{\partial D}  \nabla\varphi\cdot \eta\underbrace{\nabla^\perp\varphi_1\cdot \eta }_{=0}(\nabla \varphi\cdot \eta) d\sigma=0.
\end{align*}
Therefore, we proved 
\begin{align*}
     \dfrac{d}{dt}\Vert \varphi\Vert_V^2=2\sum_{i=1,2}\int_D\partial_i \varphi\nabla^\perp\partial_i\varphi_1\cdot \nabla \varphi dx \leq 2\int_D\vert \mathcal{D}^2\varphi_1\vert  \vert \nabla \varphi \vert^2dx,
 \end{align*}
 where $\mathcal{D}^2\varphi_1$ denotes the Hessian matrix of $\varphi_1.$ Recall that $W^{2,p}(D) \hookrightarrow W^{1,\infty}(D) $ for $p>2.$ Hence for any $\varepsilon\in ]0,1/2[$, we get (after applying Hölder inequality)
\begin{align*}
     \dfrac{d}{dt}\Vert \varphi\Vert_V^2 &\leq 2 \Vert  \nabla \varphi\Vert_\infty^\varepsilon\int_D\vert \mathcal{D}^2\varphi_1\vert  \vert \nabla \varphi \vert^{2-\varepsilon}dx \leq 2 \Vert  \nabla \varphi\Vert_\infty^\varepsilon\Vert \mathcal{D}^2\varphi_1\Vert_{L^{2/\varepsilon}}  \big(\Vert \nabla \varphi \Vert_2^2\big)^{1-\varepsilon/2} \\
     &\leq  2 \Vert  \nabla \varphi\Vert_\infty^\varepsilon\Vert \mathcal{D}^2\varphi_1\Vert_{L^{2/\varepsilon}}  \big(\Vert \varphi \Vert_V^2\big)^{1-\varepsilon/2} \leq \mathbf{K} \Vert \mathcal{D}^2\varphi_1\Vert_{L^{2/\varepsilon}}  \big(\Vert \varphi \Vert_V^2\big)^{1-\varepsilon/2}\\
     &\leq  \mathbf{K} \Vert \mathcal{D}^2\varphi_1\Vert_{L^{2/\varepsilon}}  \big(\Vert \varphi \Vert_V^2\big)^{1-\varepsilon/2},
 \end{align*}
 where $2 \Vert  \nabla \varphi\Vert_\infty^\varepsilon=\mathbf{K}>0.$ Now, we recall the   following Calderon-Zygmund’s inequality
 (see \textit{e.g.} \cite[Thm. 2.1]{Yudovich1963}), there exists $C_D>0$ depends only on $D$ such that 
\begin{align}\label{ineq-growth-p}
    \Vert \varphi\Vert_{W^{2,p}} \leq C_Dp\Vert\varphi-\Delta \varphi \Vert_p \text{ for any } 2\leq p<+\infty.
\end{align}
 By using  \eqref{inequ-p-uniq-*} and   \eqref{ineq-growth-p}, we obtain 
 \begin{align}
     \notag  \dfrac{d}{dt}\Vert \varphi\Vert_V^2 &\leq   \mathbf{K}K_*\dfrac{2}{\varepsilon} C_DK_*(\Vert \varphi_0-\Delta\varphi_0\Vert_{2/\varepsilon}+\Vert \log (n_0)\Vert_{2/\varepsilon})  \big(\Vert \varphi \Vert_V^2\big)^{1-\varepsilon/2}\\
\label{est-uniq-inequ-infty}      & \leq  \mathbf{M} \dfrac{2}{\varepsilon} (\Vert \varphi_0-\Delta\varphi_0\Vert_{2/\varepsilon}+\Vert \log (n_0)\Vert_{2/\varepsilon})  \big(\Vert \varphi \Vert_V^2\big)^{1-\varepsilon/2}\\
 \text{ (by using \eqref{assumption-uniquness}) }     & \notag\leq  2\mathbf{AM} \dfrac{2}{\varepsilon}\theta(\dfrac{2}{\varepsilon} )  \big(\Vert \varphi \Vert_V^2\big)^{1-\varepsilon/2}:= \lambda\dfrac{2}{\varepsilon}\theta(\dfrac{2}{\varepsilon} )  \big(\Vert \varphi \Vert_V^2\big)^{1-\varepsilon/2}.
 \end{align}
 Summarizing, there exists $\lambda>0$ such that
$ \dfrac{d}{dt}\Vert \varphi\Vert_V^2 \leq \lambda \Vert \varphi\Vert_V^2 \dfrac{2}{\varepsilon}\theta(\dfrac{2}{\varepsilon} )  \big(\Vert \varphi \Vert_V^2\big)^{-\varepsilon/2}
$ and we finally get
  \begin{align}\label{key-esti-1-uniq}
      \dfrac{d}{dt}\Vert \varphi\Vert_V^2 \leq \lambda \Vert \varphi\Vert_V^2\Phi_\theta(\dfrac{1}{\Vert \varphi\Vert_V^2}).
 \end{align}
 By using \eqref{Osgood-condition} and  \eqref{key-esti-1-uniq},  Osgood’s  theorem (see \textit{e.g.} \cite{Hartman2002ordinary}) ensures $\Vert \varphi (t)\Vert_V \equiv 0$ for any $t\in [0,T]$.

      \section{Proof of \autoref{TH-existence-HM-bounded-section0} (bounded density)}\label{Section-bounded-density-result}
      We prove  the  well-posedness of Hasegawa-Mima equation \eqref{HM-eqn} with bounded density.
   Denote  $g:=\log (n_0) \in L^\infty(D) \hookrightarrow L^p(D)$ for any $1\leq p< +\infty$.
         Thanks to \autoref{lemma-approximation}, there exists a sequence  $(g_\varepsilon)_\varepsilon \subseteq W$ such that $\displaystyle\lim_{\varepsilon \to 0 }g_\varepsilon=g$ in $L^p(D)$ and such that  \eqref{properties-appro}  and \eqref{Linfty-bound-appro} hold.  Set
         \begin{align}\label{Linfty-bound}
             K= \Vert \varphi_0-\Delta \varphi_0\Vert_\infty+\Vert \log(n_0)\Vert_\infty,
         \end{align}
         and note that (thanks to \eqref{Linfty-bound-appro})
         \begin{align}\label{bound-initial-data+densi}
     \text{for any } \varepsilon >0: \quad    -K\leq     \varphi_0-\Delta \varphi_0+g_\epsilon \leq K \text{  a.e. in } D.
         \end{align}

Let $\epsilon>0$, by using \autoref{THM1}, there exists  a unique   $ \varphi_\epsilon \in C([0,T];W) \cap L^2(0,T;Y)  $ such that for any $\psi \in V$ we have
     \begin{align}\label{DEF-form-epsilon-bound}
 &       \langle \partial_t(\varphi_\epsilon-\Delta\varphi_\epsilon),\psi\rangle_{V^\prime,V}+\int_D[\varepsilon\nabla(\varphi_\epsilon-\Delta\varphi_\epsilon+g_\epsilon)\cdot \nabla \psi +\nabla(g_\epsilon+\varphi_\varepsilon-\Delta \varphi_\epsilon)\cdot\nabla^\perp\varphi_\epsilon \psi]dx=0.
    \end{align}
Let $\delta>0$ and consider the  approximation ( see \textit{e.g.}  \cite[p. 152]{Pardouxthesis})
\begin{eqnarray}\label{apprx-negative-part}
F_\delta(r)= \left\{ \begin{array}{l}
r^2-\frac{\delta^2}{6} \quad\quad   \text{ if } \hspace*{0.2cm} r\leq -\delta, \\
-\frac{r^4}{2\delta^2}-\frac{4r^3}{3\delta}\quad \text{ if } \hspace*{0.2cm}  -\delta\leq r \leq 0, \\
0 \quad\quad\quad\quad \quad\quad \text{if } \hspace*{0.2cm}  r\geq 0.
\end{array}
\right.
\end{eqnarray} 
 Note that $ (F_\delta)_\delta $   satisfies 
$\displaystyle\lim_{\delta\to 0} F_\delta(\lambda)=F (\lambda)=(\lambda^-)^2$ point-wise. Moreover, $F_\delta(\cdot)\in \mathcal{C}^2(\mathbb{R})$, and satisfies:
\begin{eqnarray*}
 \left\{ \begin{array}{l}
\vert F_\delta(r)\vert \leq r^2, \qquad
\vert F_\delta^\prime(r)\vert \leq 2r \quad and \quad \forall r \in \mathbb{R}, F_\delta^\prime(r) \leq 0,\\
\vert F_\delta^{\prime\prime}(r)\vert \leq \frac{8}{3} \quad and \quad \forall r \in \mathbb{R},F_\delta^{\prime\prime}(r) \geq 0.
\end{array}
\right.
\end{eqnarray*} 

Note that $\psi=F_\delta^\prime\big(K-(\varphi_\epsilon-\Delta\varphi_\epsilon+g_\epsilon)\big) \in V$ and by using \eqref{DEF-form-epsilon-bound}, we get
   \begin{align*}
 &       \langle \partial_t(K-\varphi_\epsilon+\Delta\varphi_\epsilon),F_\delta^\prime\big(K-(\varphi_\epsilon-\Delta\varphi_\epsilon+g_\epsilon)\big)\rangle_{V^\prime,V}\\&\quad+\int_D\varepsilon\overbrace {\vert\nabla(K-\varphi_\epsilon+\Delta\varphi_\epsilon-g_\epsilon)\vert^2 F_\delta^{\prime\prime}\big(K-(\varphi_\epsilon-\Delta\varphi_\epsilon+g_\epsilon)\big)}^{\geq 0}dx\\
 &+\int_D\nabla(K-\varphi_\epsilon+\Delta\varphi_\epsilon-g_\epsilon)\cdot\nabla^\perp\varphi_\epsilon  F_\delta^\prime\big(K-(\varphi_\epsilon-\Delta\varphi_\epsilon+g_\epsilon)\big)dx=0.
 \end{align*}
 By using an integration by part formula in time (see \textit{e.g.} \cite[Lem. 5]{Tahraoui2019tools}), we get
 \begin{align*}
   \langle \partial_t(K-\varphi_\epsilon+\Delta\varphi_\epsilon),F_\delta^\prime\big(K-(\varphi_\epsilon-\Delta\varphi_\epsilon+g_\epsilon)\big)\rangle_{V^\prime,V}=\dfrac{d}{dt}\int_D F_\delta\big(K-(\varphi_\epsilon-\Delta\varphi_\epsilon+g_\epsilon)\big) dx.
 \end{align*}
Next, note that $F_\delta\big(K-(\varphi_\epsilon-\Delta\varphi_\epsilon+g_\epsilon)\big)=0$ on $\partial D$ since $g_\epsilon+\varphi_\varepsilon-\Delta \varphi_\epsilon=0$ on $\partial D$ and
 \begin{align*} 
&\int_D\nabla(K-\varphi_\epsilon+\Delta\varphi_\epsilon-g_\epsilon)\cdot\nabla^\perp\varphi_\epsilon  F_\delta^\prime\big(K-(\varphi_\epsilon-\Delta\varphi_\epsilon+g_\epsilon)\big)dx\\
&=\int_D\nabla^\perp\varphi_\epsilon   \cdot \nabla F_\delta\big(K-(\varphi_\epsilon-\Delta\varphi_\epsilon+g_\epsilon)\big)dx\\
   & =-     \int_{\partial D}\nabla^\perp\varphi_\epsilon \cdot  \eta F_\delta\big(K-(\varphi_\epsilon-\Delta\varphi_\epsilon+g_\epsilon)\big)d\sigma=0.
 \end{align*}
Consequently, we get   \begin{align*}
\int_D F_\delta\big(K-(\varphi_\epsilon(t)-\Delta\varphi_\epsilon(t)+g_\epsilon)\big) dx&\leq  \int_D F_\delta\big(K-(\varphi_0-\Delta\varphi_0+g_\epsilon)\big) dx.
 \end{align*}
 Now, we use  dominated convergence theorem, as $\delta\to 0$ to deduce for any $t\in [0,T]$
 \begin{align*}
\int_D [\big(K-(\varphi_\epsilon(t)-\Delta\varphi_\epsilon(t)+g_\epsilon)\big)^-]^2 dx&\leq  \int_D [\big(K-(\varphi_0-\Delta\varphi_0+g_\epsilon)\big)^-]^2 dx=0,
 \end{align*}
thanks to \eqref{bound-initial-data+densi}. Thus,  a.e. $x\in D$, for any $t\in [0,T]$
\begin{align}\label{uper-bound-1}
    \varphi_\epsilon(t)-\Delta\varphi_\epsilon(t)+g_\epsilon\leq K.
\end{align}
By  repeating the same arguments with $\psi=F_\delta^\prime\big(K+(\varphi_\epsilon-\Delta\varphi_\epsilon+g_\epsilon)\big) $ and combining with \eqref{uper-bound-1}, we  deduce
\begin{align}\label{lowe-uper-bound}
  \text{a.e. } x\in D, \text{ for any }t\in [0,T]: \quad -K\leq   \varphi_\epsilon(t)-\Delta\varphi_\epsilon(t)+g_\epsilon\leq K.
\end{align}
Since $\varphi_\varepsilon-\Delta \varphi_\varepsilon$ is measurable function, the last inequalities \eqref{lowe-uper-bound} ensure
\begin{align}\label{lowe-uper-bound-final}
  \Vert \varphi_\varepsilon-\Delta \varphi_\varepsilon\Vert_\infty:=  \sup_{(t,x)\in[0,T]\times D}   \vert \varphi_\epsilon(t,x)-\Delta\varphi_\epsilon(t,x)\vert \leq K+ \Vert \log(n_0)\Vert_\infty < 2K.
\end{align}
Similarly to I. and II.  in \autoref{SubSection-proof1}, we can prove
\begin{align}
 (\varphi_\varepsilon)_\varepsilon&\text{ is bounded in }    C([0,T],V\cap W^{2,p}(D)),\quad (\partial_t\varphi_\varepsilon)_\varepsilon \text{ is bounded in }
 L^\infty(0,T;L^2(D)),\label{compactness-est-1_epsilon-MOD}\end{align}
 for any $1<p<+\infty.$
  By using  \eqref{lowe-uper-bound-final} and \eqref{compactness-est-1_epsilon-MOD}, we argues as  III. in  \autoref{SubSection-proof1} to get the existence of subsequence of $(\varphi_\varepsilon)_\varepsilon$ denoted by the same way  and $\varphi \in L^\infty(0,T;V\cap W^{2,p}(D))$ such that the following convergences, as $\varepsilon \to 0,$ hold
\begin{align}
    \varphi_\varepsilon &\buildrel\ast\over\rightharpoonup \varphi \text{ in } L^\infty(0,T;V\cap W^{2,p}(D)),\quad    1<p<+\infty\label{cv1-eps-mod}\\
    \partial_t \varphi_\varepsilon &\buildrel\ast\over\rightharpoonup \partial_t\varphi \text{ in } L^2(0,T;L^2(D)), \label{cv2-eps-mod}\\
     \varphi_\varepsilon & \to \varphi \text{ in } C([0,T];W^{1,q}(D)),\quad  \forall q<+\infty, \label{cv3-eps-mod}\\
     \varphi_\varepsilon-\Delta \varphi_\varepsilon  &\buildrel\ast\over\rightharpoonup   \varphi-\Delta \varphi  \text{ in } L^\infty(D\times [0,T]) \label{cv4-infty}.
\end{align}
Indeed, \eqref{cv1-eps-mod}-\eqref{cv3-eps-mod} are similar to \eqref{cv1-eps}-\eqref{cv3-eps}. By  Banach–Alaoglu theorem there exists $\xi \in L^\infty(D\times [0,T])$ such that (for  subsequence denoted by the same way) $\varphi_\varepsilon-\Delta \varphi_\varepsilon  \buildrel\ast\over\rightharpoonup   \xi  \text{ in } L^\infty(D\times [0,T])$ and by diagonal argument, we get $\xi= \varphi-\Delta \varphi$. Finally, we pass to the limit in \eqref{DEF-form-epsilon-bound}, as  III. in \autoref{SubSection-proof1}, to conclude the existence proof.
\subsubsection*{Uniqueness}
     Let $\varphi_1$ and $\varphi_2$ be two solution to \eqref{Def-sol_HM-2} associated to the same data $(\log(n_0),\varphi_0)$ and denote by $\varphi=\varphi_1-\varphi_2$. Then, from \eqref{est-uniq-inequ-infty}  there exists $\mathbf{M}>0$ such that
\begin{align}
     \notag  \dfrac{d}{dt}\Vert \varphi\Vert_V^2 
      & \leq  \mathbf{M} \dfrac{2}{\varepsilon} (\Vert \varphi_0-\Delta\varphi_0\Vert_{2/\varepsilon}+\Vert \log (n_0)\Vert_{2/\varepsilon})  \big(\Vert \varphi \Vert_V^2\big)^{1-\varepsilon/2}
 \end{align}
  for any $\varepsilon\in ]0,1/2[$. Since $L^\infty(D) \hookrightarrow L^{2/\varepsilon}(D)$, the latter inequality ensures the existence of $\mathbf{B}>0$ (depends on the last embedding) such that
  \begin{align}
     \notag  \dfrac{d}{dt}\Vert \varphi\Vert_V^2 
      & \leq  \mathbf{MB} \dfrac{2}{\varepsilon} (\Vert \varphi_0-\Delta\varphi_0\Vert_{\infty}+\Vert \log (n_0)\Vert_{\infty})  \big(\Vert \varphi \Vert_V^2\big)^{1-\varepsilon/2}\leq   K\mathbf{MB} \dfrac{2}{\varepsilon}   \big(\Vert \varphi \Vert_V^2\big)^{1-\varepsilon/2},
 \end{align}
 which ensures (after integration with respect to $t$)
 \begin{align}
     \Vert \varphi (t)\Vert_V^2  \leq (K\mathbf{MB}t)^{\frac{2}{\varepsilon}} \to 0 \text{ as } \varepsilon \to 0,
 \end{align}
 provided that $K\mathbf{MB}t <1$ (\text{e.g.} $t\leq \frac{1}{3 K\mathbf{MB}}$). Thus $\Vert \varphi (t)\Vert_V^2=0$ on $[0,\frac{1}{3 K\mathbf{MB}}]$ and the uniqueness holds globally by repeating the argument of subintervals with length $\frac{1}{3 K\mathbf{MB}}.$
 \section{Proof of \autoref{thm-p-less2} ($p$-integrable density, $\frac{4}{3}<p<2$)}\label{Section-p-less-2}
First, we prove the following result, which serves for the proof of \autoref{thm-p-less2}.
  \begin{lemma}\label{lemma-approximation-p-less-2}
Let $1< p <2.$
  For any $h\in L^p(D)$, there exists $(h_\delta)_{\delta>0}  \subseteq  W$ such that   $h_\delta \to h$ in $L^p(D)$ as $\delta \to 0.$ Moreover,  the following holds
   \begin{align}\label{properties-appro-p-less-2}
  \text{ for any } \delta >0 : \quad \Vert  h_\delta\Vert_p^p\leq \Vert h\Vert_p^p .    
    \end{align}
  \end{lemma}
\begin{proof}
Let $h\in L^p(D)$ and $\delta>0$, then $T_{1/\delta}(h)\footnote{$T_k(x)=\max(\min(x,k),-k),\quad  x\in \mathbb{R}, \quad k\in \mathbb{R}^+.$}\in L^\infty(D)\hookrightarrow L^2(D)$ and $\displaystyle\lim_{\delta \to 0}T_{1/\delta}(h)=h$  in $L^p(D)$ thanks to  Lebesgue's dominated convergence theorem. By classical arguments (Lax-Milgram theorem and elliptic regularity theory, see  \textit{e.g.} \cite[Thm. 9.25]{BrezisBook}) there exists a unique $h_\delta \in H^2(D)\cap H^1_0(D)=W$ such that 
    \begin{align}\label{form-Lax-Mil-V2}
     \int_D h_\delta\psi dx+\delta\int_D \nabla h_\delta\cdot \nabla \psi dx=\int_D T_{1/\delta}(h)\psi dx ,  \quad \forall \psi \in H^1_0(D).  
    \end{align}
Let $\gamma>0,$    set $\psi=(\vert h_\delta\vert + \gamma)^{p-2}h_\delta$ in \eqref{form-Lax-Mil-V2} and use Young inequality   to get 
    \begin{align*}
          \int_D h_\delta^2(\vert h_\delta\vert + \gamma)^{p-2} dx&+\delta\int_D \vert \nabla h_\delta\vert^2\big( [(p-1)\vert h_\delta\vert+\gamma](\vert h_\delta\vert + \gamma)^{p-3}\big)dx\\&=\int_D T_{1/\delta}(h)h_\delta(\vert h_\delta\vert + \gamma)^{p-2} dx  \leq \dfrac{1}{p} \Vert h\Vert_p^p+\dfrac{p-1}{p}\Vert (\vert h_\delta\vert + \gamma)\Vert_p^p.
    \end{align*}
    Lebesgue's dominated convergence theorem ensures, as $\gamma \to 0$
    \begin{align}\label{est-delta-pless2}
    \text{ for any } \delta >0: \quad    \Vert h_\delta\Vert_p^p\leq  \Vert h\Vert_p^p.    
    \end{align}
           It follows from \eqref{form-Lax-Mil-V2},  \eqref{est-delta-pless2} and the density of $H^2(D)\cap H^1_0(D)$ in $L^{p/p-1}(D)$  that $h_\delta \rightharpoonup h$ in $L^p(D)$ as $\delta \to 0.$ Now, \eqref{est-delta-pless2} ensures that $\displaystyle\limsup_{\delta \to 0} \Vert  h_\delta\Vert_p^p\leq \Vert h\Vert_p^p$ and thus $\Vert h_\delta - h\Vert_p \to 0$ as $\delta \to 0$ by uniform convexity argument. 
\end{proof}

    Let $1< p<2$,  $g:=\log (n_0) \in L^p(D)$ and      $\varphi_0-\Delta \varphi_0\in L^p(D)$ such that $\varphi_0=0$ on $\partial D.$ \\

    Thanks to \autoref{lemma-approximation-p-less-2}, there exists a sequence  $(g_\varepsilon)_\varepsilon \subseteq W$ such that $\displaystyle\lim_{\varepsilon \to 0 }g_\varepsilon=g$ in $L^p(D)$ and such that  \eqref{properties-appro-p-less-2} holds. 
    On the other hand, by density  argument there exists a sequence $(\varphi_\varepsilon^0)_\varepsilon \subseteq Y$ such that 
    \begin{align}\label{approx-initial-data-case3-p}
     \varphi_\varepsilon^0 \to \varphi_0  \text{ in }  H^1_0(D)\cap W^{2,p}(D)  \text{ as } \epsilon \to 0.
    \end{align}
Let $\epsilon>0$, by using \autoref{THM1}, there exists  a unique   $ \varphi_\epsilon \in C([0,T];W) \cap L^2(0,T;Y)  $ such that for any $\psi \in V$ we have
     \begin{align}\label{DEF-form-epsilon-3case}
 &       \langle \partial_t(\varphi_\epsilon-\Delta\varphi_\epsilon),\psi\rangle_{V^\prime,V}+\int_D[\varepsilon\nabla(\varphi_\epsilon-\Delta\varphi_\epsilon+g_\epsilon)\cdot \nabla \psi -(g_\epsilon-\Delta \varphi_\epsilon)\nabla^\perp\varphi_\epsilon\cdot \nabla \psi]dx=0.
    \end{align}
Let $p\in ]1,2[,$ $\delta>0$ and denote by $\beta_\delta$ the even $C^1$-function defined on $\lambda\in [0,+\infty[$ by
\begin{align}\label{aprox-norm-p-2case}
    \beta_\delta(\lambda)&=\big(\lambda^p-\frac{4-p}{4}(2\delta)^p\big)1_{[2\delta,+\infty[}(\lambda)\\&\qquad+\big( p(2\delta)^{p-2}(\lambda-2\delta)^2+p(2\delta)^{p-1}(\lambda-2\delta)+\frac{p}{4}(2\delta)^p\big)1_{[\delta,2\delta[}(\lambda).\notag
\end{align}
This is a  sequence of non-negative functions such that $\displaystyle\lim_{\delta\to 0}\beta_\delta(\lambda)=\beta (\lambda)=\vert \lambda\vert^p$ point-wise, $\vert\beta_\delta(\lambda)\vert \leq 1+\vert \lambda\vert^p$. Moreover, it satisfies  $\beta_\delta^{\prime\prime}\geq 0$ a.e. and $\beta_\delta^{\prime\prime} \in L^\infty(\mathbb{R}).$\\

Note that $\psi=\beta_\delta^\prime(\varphi_\epsilon-\Delta\varphi_\epsilon+g_\epsilon) \in V, \beta_\delta(\varphi_\epsilon-\Delta\varphi_\epsilon+g_\epsilon)=0$ on $\partial D$ and by using \eqref{DEF-form-epsilon-3case}, we get

   \begin{align*}
 &       \langle \partial_t(\varphi_\epsilon-\Delta\varphi_\epsilon),\beta_\delta^\prime(\varphi_\epsilon-\Delta\varphi_\epsilon+g_\epsilon)\rangle_{V^\prime,V}+\int_D\varepsilon\underbrace {\vert\nabla(\varphi_\epsilon-\Delta\varphi_\epsilon+g_\epsilon)\vert^2 \beta_\delta ^{\prime\prime}(\varphi_\epsilon-\Delta\varphi_\epsilon+g_\epsilon)}_{\geq 0}dx\\
 &-\int_D(g_\epsilon-\Delta \varphi_\epsilon)\nabla^\perp\varphi_\epsilon\cdot \nabla \beta_\delta^\prime(\varphi_\epsilon-\Delta\varphi_\epsilon+g_\epsilon)dx=0.
 \end{align*}
 By using an integration by part formula in time (see \textit{e.g.} \cite[Lem. 5]{Tahraoui2019tools}), we get
 \begin{align*}
     \langle \partial_t(\varphi_\epsilon-\Delta\varphi_\epsilon),\beta_\delta^\prime(\varphi_\epsilon-\Delta\varphi_\epsilon+g_\epsilon)\rangle_{V^\prime,V}&=\langle \partial_t(\varphi_\epsilon-\Delta\varphi_\epsilon+g_\epsilon),\beta_\delta^\prime(\varphi_\epsilon-\Delta\varphi_\epsilon+g_\epsilon)\rangle_{V^\prime,V}\\
     &=\dfrac{d}{dt}\int_D \beta_\delta(\varphi_\epsilon-\Delta\varphi_\epsilon+g_\epsilon) dx.
 \end{align*}
 On the other hand, by using that $g_\epsilon-\Delta \varphi_\epsilon=0$ on $\partial D$ and  $\nabla^\perp\varphi_\epsilon \cdot \nabla \varphi_\varepsilon=0$, we get 
 \begin{align*}
     \int_D(g_\epsilon-\Delta \varphi_\epsilon)\nabla^\perp\varphi_\epsilon\cdot \nabla \beta_\delta^\prime(\varphi_\epsilon-\Delta\varphi_\epsilon+g_\epsilon)dx&=-     \int_D\nabla^\perp\varphi_\epsilon \cdot \nabla(g_\epsilon+\varphi_\varepsilon-\Delta \varphi_\epsilon)\beta_\delta^\prime(\varphi_\epsilon-\Delta\varphi_\epsilon+g_\epsilon)dx\\
    & =-     \int_D\nabla^\perp\varphi_\epsilon \cdot \nabla\beta_\delta(\varphi_\epsilon-\Delta\varphi_\epsilon+g_\epsilon)dx\\
   & =-     \int_{\partial D}\nabla^\perp\varphi_\epsilon \cdot  \eta \beta_\delta(\varphi_\epsilon-\Delta\varphi_\epsilon+g_\epsilon)d\sigma=0,
 \end{align*}
 since $\beta_\delta(\varphi_\epsilon-\Delta\varphi_\epsilon+g_\epsilon)=0$ on $\partial D.$ Therefore we get  \begin{align*}
\int_D \beta_\delta(\varphi_\epsilon(t)-\Delta\varphi_\epsilon(t)+g_\epsilon) dx&\leq  \int_D \beta_\delta(\varphi_\epsilon^0-\Delta\varphi_\epsilon^0+g_\epsilon) dx.
 \end{align*}
 Now, we use      Lebesgue's dominated convergence theorem, as $\delta\to 0$ to deduce for any $t\in [0,T]$
 \begin{align}\label{ineq-lowe-semi-p-less-2}
\int_D \vert\varphi_\epsilon(t)-\Delta\varphi_\epsilon(t)+g_\epsilon\vert^p dx&\leq \int_D \vert \varphi_\epsilon^0-\Delta\varphi_\epsilon^0+g_\epsilon\vert^p dx 
\leq  2^{p-1}(\Vert \varphi_\epsilon^0-\Delta\varphi_\epsilon^0\Vert_p^p+\Vert g_\epsilon\Vert_p^p).
 \end{align}
 Thus, for any $t\in [0,T]$, we have
 \begin{align*}
\int_D \vert\varphi_\epsilon(t)-\Delta\varphi_\epsilon(t)\vert^p dx \leq  4^{p-1}(\Vert \varphi_0-\Delta\varphi_0\Vert_p^p+1+2\Vert g\Vert_p^p)
 \end{align*}
 by using \eqref{properties-appro-p-less-2} and the convergence of $\varphi_\varepsilon^0 \to \varphi_0$ in $W^{2,p}(D)$. Consequently, we proved
 \begin{align}\label{bound-p-pot-pless2}
     (\varphi_\epsilon)_\varepsilon \text{ is bounded uniformly w.r.t. } \varepsilon  \text{ in } C([0,T],V\cap W^{2,p}(D)).
 \end{align}
 \subsubsection*{II. Uniform estimates of $(\partial_t\varphi_\varepsilon)_\varepsilon$ with respect to $\varepsilon$} Let $\dfrac{4}{3}\leq p<2$, 
 denote by $q$ the conjugate of $p$ and 
 note that    $$\mathcal{X}:=W^{2,q}(D)\cap V\hookrightarrow V\hookrightarrow L^2(D) \hookrightarrow V^\prime\hookrightarrow (W^{2,q}(D)\cap V)^\prime.$$ Thus,
$\partial_t(\varphi_\varepsilon-\Delta\varphi_\varepsilon) \in L^2(0,T;(W^{2,q}(D)\cap V)^\prime).$ Now, let $v\in W^{2,q}(D)\cap V$ and note that 
\begin{align*}
   \langle  \partial_t(\varphi_\varepsilon-\Delta\varphi_\varepsilon),v\rangle_{\mathcal{X}^\prime,\mathcal{X}}&:=   \langle  \partial_t(\varphi_\varepsilon-\Delta\varphi_\varepsilon),v\rangle_{V^\prime,V}.
\end{align*}
By using \eqref{DEF-form-epsilon-3case} and since $\varphi_\varepsilon \in L^2(0,T;Y)$, we get
\begin{align*}
    &  \langle  \partial_t(\varphi_\varepsilon-\Delta\varphi_\varepsilon),v\rangle_{\mathcal{X}^\prime,\mathcal{X}}=-\int_D\nabla^\perp\varphi_\varepsilon\cdot \nabla  v  \Delta \varphi_\varepsilon dx+\int_D\nabla^\perp\varphi_\varepsilon\cdot \nabla  v g_\varepsilon  dx+\varepsilon\int_D (\varphi_\varepsilon-\Delta\varphi_\varepsilon+g_\varepsilon)\Delta v  dx.
\end{align*}
Recall that $q=\dfrac{p}{p-1}>2$ and that $\mathcal{X}\hookrightarrow W^{1,\infty}(D)$.  
 On the other hand, note that 
 $$ W^{2,p}(D) \hookrightarrow  W^{1,p^*}(D) \text{ where } p^*=\dfrac{2p}{2-p},  \quad 1<p<2.$$
If  $p\geq \dfrac{4}{3}$ then  $\dfrac{1}{p}+\dfrac{1}{p^*}=1.$
Therefore, Hölder's inequality ensures
\begin{align*}
       \vert \langle \partial_t(\varphi_\varepsilon-\Delta\varphi_\varepsilon),v\rangle_{\mathcal{X}^\prime,\mathcal{X}}\vert&\leq \Vert\nabla^\perp\varphi_\varepsilon\Vert_{p^*}\Vert  \Delta \varphi_\varepsilon\Vert_p \Vert \nabla v \Vert_\infty+\Vert \nabla^\perp\varphi_\varepsilon\Vert_{p^*}\Vert g_\varepsilon\Vert_p\Vert \nabla v \Vert_\infty\\&\quad+\varepsilon\Vert \varphi_\varepsilon-\Delta\varphi_\varepsilon+g_\varepsilon\Vert_p \Vert \Delta v\Vert_q.
\end{align*}
Consequently, by using \eqref{properties-appro-p-less-2} we get
 \begin{align}\label{est-time-deriv-epsilon-pless2}
 \notag   \vert \langle \partial_t(\varphi_\varepsilon-\Delta\varphi_\varepsilon),v\rangle_{\mathcal{X}^\prime,\mathcal{X}}\vert &\leq M\big(\Vert\varphi_\varepsilon\Vert_{W^{2,p}}^2+\Vert \varphi_\varepsilon\Vert_{W^{2,p}}\Vert g_\varepsilon\Vert_p+\varepsilon\Vert \varphi_\varepsilon-\Delta\varphi_\varepsilon\Vert_p+ \varepsilon\Vert g_\varepsilon\Vert_p\big)\Vert v \Vert_\mathcal{X}\\
    &\leq M\big(\Vert\varphi_\varepsilon\Vert_{W^{2,p}}^2+\Vert \varphi_\varepsilon\Vert_{W^{2,p}}\Vert g\Vert_p+\varepsilon\Vert \varphi_\varepsilon\Vert_{W^{2,p}}+ \varepsilon\Vert g\Vert_p\big)\Vert v \Vert_\mathcal{X}
\end{align}
where $M>0$ depending only on the embeddings $\mathcal{X}\hookrightarrow W^{1,\infty}(D)$. Consequently, 
\eqref{est-time-deriv-epsilon-pless2} together with  \eqref{bound-p-pot-pless2} ensure the boundedness of $(\partial_t(\varphi_\varepsilon-\Delta\varphi_\varepsilon))_\varepsilon$
in $L^\infty(0,T;\mathcal{X}^\prime).$ Thus, the  boundedness of $(\partial_t\varphi_\varepsilon)_\varepsilon$
in $L^\infty(0,T;L^p(D)).$  
\subsubsection*{III. Compactness and passage to the limit as $\varepsilon\to 0$}
In the previous subsection we proved 
\begin{align}
 (\varphi_\varepsilon)_\varepsilon&\text{ is bounded in }    C([0,T],V\cap W^{2,p}(D)),\quad (\partial_t\varphi_\varepsilon)_\varepsilon \text{ is bounded in }
 L^\infty(0,T;L^p(D)).\label{compactness-est-1_epsilon-pless2}\end{align}
 We recall that $W^{2,p}(D) \hookrightarrow_c  W^{1,r}(D)$ for any $1\leq r<p^*=\dfrac{2p}{2-p}$ (
  see \textit{e.g.} \cite[Thm. 1.20]{Roubivcek2005nonlinear}). We introduce the set
  \begin{align*}
 \widetilde{B}_K=\{f\in L^\infty(0,T;V\cap W^{2,p}(D))&; \partial_t f\in  L^\infty(0,T;L^p(D))\\&:  \Vert f\Vert_{L^\infty(0,T;W^{2,p}(D))}+\Vert \partial_t f\Vert_{L^\infty(0,T;L^p(D))} \leq K \}, \quad K>0.      
  \end{align*}

Note that 
\begin{align}\label{strongcv-eps-pless2}
\widetilde{B}_K  \text{ is relatively compact in } C([0,T],W^{1,r}(D)) \text{ thanks to \cite[Cor. 4 ]{Simoncompact}}.
\end{align} Consequently, we get the existence of subsequence of $(\varphi_\varepsilon)_\varepsilon$ denoted by the same way  and $\varphi \in L^\infty(0,T;V\cap W^{2,p}(D))$ such that the following convergences, as $\varepsilon \to 0,$ hold
\begin{align}
    \varphi_\varepsilon &\buildrel\ast\over\rightharpoonup \varphi \text{ in } L^\infty(0,T;V\cap W^{2,p}(D)), \label{cv1-eps-pless2}\\
    \partial_t \varphi_\varepsilon &\buildrel\ast\over\rightharpoonup \partial_t\varphi \text{ in } L^\infty(0,T;L^p(D)), \label{cv2-eps-p-less2}\\
     \varphi_\varepsilon & \to \varphi \text{ in } C([0,T];W^{1,r}(D)),\quad  \forall r<p^*, \label{cv3-eps-pless2}.
\end{align}
Indeed, \eqref{cv1-eps-pless2} and \eqref{cv2-eps-p-less2} are consequence of Banach–Alaoglu theorem, \eqref{cv3-eps-pless2} is consequence of \eqref{strongcv-eps-pless2}. On the other hand, recall that \autoref{lemma-approximation-p-less-2} ensures
\begin{align}
    g_\varepsilon \to g \text{ in }  L^p(D) \label{cv4-eps-pless2}.
\end{align}
Moreover, \eqref{approx-initial-data-case3-p} and \eqref{cv3-eps-pless2}
ensures that 
\begin{align}
\varphi_\varepsilon(0)=\varphi_\varepsilon^0 \text{ converges  in } W^{1,r}(D) \text{ to } \varphi(0)=\varphi_0 \in V\cap W^{2,p}(D).     
\end{align}
Furthermore,  \eqref{compactness-est-1_epsilon-pless2} ensures that $\varphi_\varepsilon(t) \rightharpoonup \varphi(t)$ in  $V\cap W^{2,p}(D)$ for any $t\in [0,T]$ by an argument similar to that used to obtain \eqref{cv4}.\\

Let $\psi \in W^{2,q}(D)\cap V$ and recall that from \eqref{DEF-form-epsilon-3case} we have
\begin{align}\label{DEF-form-epsilon---p2}
 &       \langle \partial_t(\varphi_\epsilon-\Delta\varphi_\epsilon),\psi\rangle_{\mathcal{X}^\prime,\mathcal{X}}-\int_D (g_\epsilon-\Delta \varphi_\epsilon)\nabla^\perp\varphi_\epsilon\cdot \nabla \psi dx=\varepsilon\int_D(\varphi_\epsilon-\Delta\varphi_\epsilon+g_\epsilon)\Delta \psi.
    \end{align}
    Thanks to \autoref{lemma-approximation-p-less-2} and 
  \eqref{compactness-est-1_epsilon-pless2}, we get 
  \begin{align*}
      \lim_{\varepsilon \to 0} \varepsilon\int_D(\varphi_\epsilon-\Delta\varphi_\epsilon+g_\epsilon)\Delta \psi=0.
  \end{align*}
Since $\frac{4}{3}<p<2 $,  we have $L^r(D)\hookrightarrow L^{p^*}(D)\hookrightarrow L^q(D)$ and  $\Delta \varphi\nabla^\perp\varphi\in L^\infty(0,T;L^1(D)).$ Thus 
  \begin{align}
      \int_D \Delta \varphi_\epsilon\nabla^\perp\varphi_\epsilon\cdot \nabla \psi dx \to        \int_D \Delta \varphi\nabla^\perp\varphi\cdot \nabla \psi dx, \quad \text{ as } \varepsilon \to 0
  \end{align}
  since $\nabla \psi\in L^\infty(D).$ 
  Therefore, it is possible to use \eqref{cv1-eps-pless2}-\eqref{cv4-eps-pless2} and pass to the limit as $\varepsilon \to 0$ in \eqref{DEF-form-epsilon---p2} to get
\begin{align}\label{DEF-weak-solu-p-less2}
 &       \langle \partial_t(\varphi-\Delta\varphi),\psi\rangle_{\mathcal{X}^\prime,\mathcal{X}}-\int_D (g-\Delta \varphi)\nabla^\perp\varphi\cdot \nabla \psi dx=0, \quad \forall \psi \in W^{2,q}(D)\cap V.
    \end{align}
  In conclusion, we proved  the existence of $\varphi$ satisfying
    \begin{itemize}
        \item $\varphi \in  C([0,T];W^{1,r}(D)) \text{ for any } 1\leq r < \frac{2p}{2-p}  \text{ and }  \varphi \in L^\infty(0,T;V\cap W^{2,p}(D)). $
        \item $\varphi \in C([0,T];V\cap W^{2,p}(D)-w)$, see \textit{e.g.}  \cite[Lem 1.4 p. 263]{Temam2024navier}.
   \item $\varphi$ satisfies $\varphi(0)=\varphi_0$ and  solves \eqref{HM-eqn} in the following sense   
      \begin{align}\label{limlit1-V}
 &       \langle \partial_t(\varphi-\Delta\varphi),\psi\rangle_{W^\prime,W}-\int_D(g-\Delta \varphi)\nabla^\perp\varphi\cdot \nabla \psi dx=0, \quad \forall \psi \in W^{2,q}(D)\cap V.
    \end{align}  
     \end{itemize}
    By using the lower semicontinuity of the  weak convergence in $V\cap W^{2,p}(D)$, \eqref{ineq-lowe-semi-p-less-2} \autoref{lemma-approximation-p-less-2} and  \eqref{approx-initial-data-case3-p}, we get
     \begin{align}\label{inequality-uniq-Yodu-2}
\text{ for any } t\in [0,T]: \quad
\int_D \vert\varphi(t)-\Delta\varphi(t)+g\vert^p dx \leq \int_D \vert \varphi_0-\Delta\varphi_0+g\vert^p dx. 
 \end{align}
 Thus, there exists $\mathcal{K}_*>0$ independent of $p$ such that
   \begin{align*}
   \sup_{t\in[0,T]}   \Vert \varphi(t)-\Delta\varphi(t)\Vert_p \leq  \mathcal{K}_*(\Vert \varphi_0-\Delta\varphi_0\Vert_p+\Vert g\Vert_p).
 \end{align*}

\subsection*{Acknowledgements}  The research of F.F. and Y.T. is
	funded by the European Union (ERC, NoisyFluid, No. 101053472). Views and
	opinions expressed are however those of the authors only and do not necessarily
	reflect those of the European Union or the European Research Council. Neither
	the European Union nor the granting authority can be held responsible for them.
      \bibliography{HM-bibliography}{}
\bibliographystyle{plain}

\end{document}